\documentclass[11pt]{amsart}
\usepackage{amsmath, amssymb, amsthm, amsfonts}
\usepackage{comment}
\usepackage{hyperref,doi}
\usepackage{epsf}
\usepackage{enumerate}
\usepackage{graphicx}
\usepackage{array}
\usepackage[margin=1in]{geometry}
\usepackage{color}

\DeclareMathOperator{\CUT}{cut}
\DeclareMathOperator{\ord}{ord}
\DeclareMathOperator{\rank}{rank}
\DeclareMathOperator{\spn}{span}
\DeclareMathOperator{\tr}{Tr}

\newcommand{\R}{\mathbb{R}}
\newcommand{\s}{\mathbb{S}}
\newcommand{\E}{\mathcal{E}}

\newcommand{\CC}{\mathcal{C}}
\newcommand{\y}{\bf y}
\newcommand{\nonedge}{\overline{E}}
\newcommand{\symG}{\s^p_{\succeq0}(G)}
\newcommand{\cut}[1]{\operatorname{\CUT^{01}\left({#1}\right)}}
\newcommand{\cutpm}[1]{\operatorname{\CUT^{\pm1}\left({#1}\right)}}

\newcommand{\symmetricmatrices}{\s^p}

\newtheorem{thm}{Theorem}[section]
\newtheorem{prob}[thm]{Problem}

\newtheorem{prop}[thm]{Proposition}
\newtheorem{cor}[thm]{Corollary}
\newtheorem{lem}[thm]{Lemma}
\theoremstyle{definition}

\newtheorem{defn}[thm]{Definition}
\newtheorem{rmk}[thm]{Remark}
\newtheorem{example}[thm]{Example}

\begin{document}

\begin{abstract}
For a graph $G$ with $p$ vertices the closed convex cone $\symG$ consists of all 
real positive semidefinite $p\times p$ matrices with zeros in the off-diagonal entries corresponding to nonedges of~$G$.  
The extremal rays of this cone and their associated ranks have applications to matrix completion problems, 
maximum likelihood estimation in Gaussian graphical models in statistics, and Gauss elimination for sparse matrices.  
For a graph $G$ without $K_5$ minors, we show that the normal vectors 
to the facets of the $(\pm1)$-cut polytope of $G$ specify the off-diagonal entries of extremal matrices in $\symG$.  
We also prove that the constant term of the linear equation of each facet-supporting hyperplane is the rank of its corresponding extremal matrix in $\symG$.
Furthermore, we show that if $G$ is series-parallel then this gives a complete characterization of all possible extremal ranks of $\symG$, 
consequently solving the sparsity order problem for series-parallel graphs.  
\end{abstract}

\title[Extremal Positive Semidefinite Matrices for graphs without $K_5$ minors]{Extremal Positive Semidefinite Matrices\\ for graphs without $K_5$ minors} 
\author{Liam Solus}
\address{Department of Mathematics\\
         University of Kentucky\\
         Lexington, KY 40506--0027, USA}
\email{liam.solus@uky.edu}

\author{Caroline Uhler}
\address{Institute of Science and Technology Austria\\
         Am Campus I\\
         3400 Klosterneuburg, Austria\\ and\qquad
         Department of Electrical Engineering and Computer Science\\ Massachusetts Institute of Technology\\ Cambridge, MA 02139--4307, USA}
\email{caroline.uhler@ist.ac.at}

\author{Ruriko Yoshida}
\address{Department of Statistics\\
         University of Kentucky\\
         Lexington, KY 40506--0082, USA}
\email{ruriko.yoshida@uky.edu}

\date{\today}

\maketitle


\section{Introduction}
\label{introduction}

For a positive integer $p$ let $[p]:=\{1,2,\ldots,p\}$, and let $G$ be a graph with vertex set $V(G)=[p]$ and edge set $E:=E(G)$.  
To the graph $G$ we associate the closed convex cone $\symG$ consisting of all real $p \times p$ positive semidefinite matrices with zeros in all entries corresponding to the nonedges of $G$.  
In this paper, we study the problem of characterizing the possible ranks of the extremal matrices in $\symG$.  
This problem has applications to the positive semidefinite completion problem, and consequently, maximum likelihood estimation for Gaussian graphical models.  
Thus, the extreme ranks of $\symG$, and in particular the maximum extreme rank of $\symG$, have been studied extensively \cite{agler, grone, helton, laurent1}.  
However, as noted in \cite{agler} the nonpolyhedrality of $\symG$ makes this problem difficult, and as such there remain many graph classes for which the extremal ranks of $\symG$ are not well-understood.  
Our main contribution to this area of study is to show that the polyhedral geometry of a second well-studied convex body, the cut polytope of $G$, serves to characterize the extremal ranks of $\symG$ for new classes of graphs.  

The thrust of the research in this area has been focused on determining the \emph{(sparsity) order of $G$}, i.e. the maximum rank of an extremal ray of $\symG$.  
In \cite{agler} it is shown that the order of $G$ is one if and only if $G$ is a chordal graph, that is, a graph in which all induced cycles have at most three edges.  
Then in \cite{laurent1} all graphs with order two are characterized.  
In \cite{helton}, it is shown that the order of $G$ is at most $p-2$ with equality if and only if $G$ is the cycle on $p$ vertices, and in \cite{grone} the order of the complete bipartite graph is computed and it is shown that all possible extreme ranks are realized.  
However, beyond the chordal, order two, cycle, and complete bipartite graphs there are few graphs for which all extremal ranks are characterized.  
Our main goal in this paper is to demonstrate that the geometric relationship between $\symG$ and the cut polytope of $G$ can serve to expand this collection of graphs.

A \emph{cut} of the graph $G$ is a bipartition of the vertices, $(U,U^c)$, and its associated \emph{cutset} is the collection of edges $\delta(U)\subset E$ with one endpoint in each block of the bipartition.  
To each cutset we assign a $(\pm1)$-vector in $\R^E$ with a $-1$ in coordinate $e$ if and only if $e\in \delta(U)$.  
The $(\pm1)$-cut polytope of $G$ is the convex hull in $\R^E$ of all such vectors.  
The polytope $\cutpm{G}$ is affinely equivalent to the cut polytope of $G$ defined in the variables $0$ and $1$, which is the feasible region of the \emph{max-cut problem} in linear programming.  
Hence, maximizing over the polytope $\cutpm{G}$ is equivalent to solving the max-cut problem for $G$.  
The max-cut problem is known to be NP-hard \cite{poljak}, and thus the geometry of $\cutpm{G}$ is of general interest.  
In particular, the facets of $\cutpm{G}$ have been well-studied \cite[Part V]{deza}, as well as a positive semidefinite relaxation of $\cutpm{G}$, known as the \emph{elliptope of $G$} \cite{BJL96,barrett,laurent3,laurent2}.  
  
Let $\s^p$ denote the real vector space of all real $p\times p$ symmetric matrices, and let $\s^p_{\succeq0}$ denote the cone of all positive semidefinite matrices in $\s^p$.  
The \emph{$p$-elliptope} is the collection of all $p\times p$ \emph{correlation matrices}, i.e.
	$$\E_p=\{X\in\s_{\succeq0}^p : X_{ii}=1 \mbox{ for all $i\in[p]$}\}.$$
The elliptope $\E_G$ is defined as the projection of $\E_p$ onto the edge set of $G$.  
That is,
	$$\E_G=\{\y\in\R^E : \mbox{ $\exists Y\in\E_p$ such that $Y_e=y_e$ for every $e\in E(G)$}\}.$$
The elliptope $\E_G$ is a positive semidefinite relaxation of the cut polytope $\cutpm{G}$ \cite{laurent2}, and thus maximizing over $\E_G$ can provide an approximate solution to the max-cut problem. 
  
In this article we show that the facets of $\cutpm{G}$ identify extremal rays of $\symG$ for any graph $G$ that has no $K_5$ minors.  
We will see in addition that the rank of the extreme ray identified by the facet with supporting hyperplane $\langle \alpha, x\rangle = b$ has rank $b$,  and if $G$ is also series-parallel (i.e. no $K_4$ minors), then all possible ranks of extremal rays are given in this fashion.  
The method by which we will make these identifications arises via the geometric relationship that exists between the three convex bodies $\cutpm{G}$, $\E_G$, and $\symG$.  
A key component of this relationship is the following theorem which is proven in Section~\ref{the dual of an elliptope}.  
	\begin{thm}
	\label{dual body theorem}
	The polar of the elliptope $\E_G$ (see (\ref{def:polar}) for a definition) is given by 
	$$
	\E_G^\circ = \{ x\in\mathbb{R}^E : \exists \,X\in \symG \textrm{ such that } X_E=x \,\textrm{ and } \, \textrm{tr}(X)= 2\}.
	$$
	\end{thm}
An immediate consequence of Theorem~\ref{dual body theorem} is that the extreme points in $\E_G^\circ$ are projections of extreme matrices in $\symG$ (recall that a subset $F$ of a convex set $K$ is called an \emph{extreme set} of $K$ if, for all $x\in F$ and $a,b\in K$, $x = (a+b)/2$ implies $a,b\in F$; 
so an \emph{extreme point} is any point in the set that does not lie on the line segment between any two distinct points of $K$).  

With Theorem~\ref{dual body theorem} in hand, the identification of extremal rays of $\symG$ via facets of $\cutpm{G}$ is guided by the following geometry.  
Since $\E_G$ is a positive semidefinite relaxation of $\cutpm{G}$, then $\cutpm{G}\subset\E_G$.  
If all singular points on the boundary of $\E_G$ are also singular points on the boundary of $\cutpm{G}$, then the supporting hyperplanes of facets of $\cutpm{G}$ will be translations of supporting hyperplanes of \emph{regular extreme points} of $\E_G$ or \emph{facets} of $\E_G$, i.e.~extreme sets of $\E_G$ with positive Lebesgue measure in a codimension one affine subspace of the ambient space.  
It follows that the outward normal vectors to the facets of $\cutpm{G}$ generate the normal cones to these regular points and facets of $\E_G$.  
Dually, the facet-normal vectors of $\cutpm{G}$ are then extreme points of $\E_G^\circ$, and consequently projections of extreme matrices of $\symG$.  
Thus, we can expect to find extremal matrices in $\symG$ whose off-diagonal entries are given by the facet-normal vectors of $\cutpm{G}$.  
This motivates the following definition.  

	\begin{defn}
	\label{facet-ray identification property}
	Let $G$ be a graph.  
	For each facet $F$ of $\cutpm{G}$ let $\alpha^F\in\R^E$ denote the normal vector to the supporting hyperplane of $F$.  
	We say that $G$ has the {\bf facet-ray identification property} (or FRIP) if for every facet $F$ of $\cutpm{G}$ there exists an extremal matrix $M=[m_{ij}]$ in $\symG$ for which either $m_{ij}=\alpha^F_{ij}$ for every $\{i,j\}\in E(G)$ or $m_{ij}=-\alpha^F_{ij}$ for every $\{i,j\}\in E(G)$.
	\end{defn}

An explicit example of facet-ray identification and its geometry is presented in Section~\ref{four cycle example}.  
With this example serving as motivation, our main goal is to identify interesting collections of graphs exhibiting the facet-ray identification property.  
Using the combinatorics of cutsets as well as the tools developed by Agler et~al. in \cite{agler}, we will prove the following theorem in Section~\ref{facet-ray identification for the cycle}.
\begin{thm}
\label{main theorem graphs with no K_5 minors}
Graphs without $K_5$ minors have the facet-ray identification property.
\end{thm}

Recall that a cycle subgraph of a graph $G$ is called \emph{chordless} if it is an induced subgraph of $G$.  
For graphs without $K_5$ minors the facet-defining hyperplanes of $\cutpm{G}$ are of the form $\langle \alpha, x\rangle = b$, where $b=1$ or  $b = p-2$ for $C_p$ a chordless cycle of $G$ \cite{BM86}.   
In \cite{agler}, it is shown that the $p$-cycle $C_p$ is a \emph{$(p-2)$-block}, meaning that if $C_p$ is an induced subgraph of $G$, then the sparsity order of $G$ is at least $p-2$.  
Since the facets of $\cutpm{G}$ are given by the chordless cycles in $G$, then Theorem~\ref{main theorem graphs with no K_5 minors} demonstrates that this condition arises via the geometry of the cut polytope $\cutpm{G}$.  
That is, since the elliptope $\E_G$ is a positive semidefinite relaxation of $\cutpm{G}$ we can translate the facet-supporting hyperplanes of $\cutpm{G}$ to support points on $\E_G$.  
By Theorem~\ref{dual body theorem} these supporting hyperplanes correspond to points in $\E_G^\circ$, and by Theorem~\ref{main theorem graphs with no K_5 minors} we see that these points are all extreme.  
In this way, the lower bound on sparsity order of $G$ given by the chordless cycles is a consequence of the relationship between the chordless cycles and the facets of $\cutpm{G}$.  
In the case that $G$ is a series-parallel graph, we will prove in Section~\ref{sec:series-parallel} that the facets of $\cutpm{G}$ in fact determine \emph{all} possible extremal ranks of $\symG$.  
\begin{thm}
	\label{main theorem characterization of all extremal ranks}
	Let $G$ be a series-parallel graph. 
	The constant terms of the facet-defining hyperplanes of $\cutpm{G}$ characterize the ranks of extremal rays of $\symG$.  
	These ranks are $1$ and $p-2$ where $C_p$ is any chordless cycle in $G$.
	Moreover, the sparsity order of $G$ is $p^*-2$ where $p^*$ is the length of the largest chordless cycle in $G$.  
	\end{thm}
The remainder of this article is organized as follows:  
In Section~\ref{preliminaries} we recall some of the previous results on sparsity order and cut polytopes that will be fundamental to our work.  
Then in Section~\ref{the geometric correspondence}, we describe the geometry underlying the facet-ray identification property.
We begin the section with the motivating example of the $4$-cycle, in which we explicitly illustrate the geometry described above.  
We then provide a proof of Theorem~\ref{dual body theorem} and discuss how this result motivates the definition of the facet-ray identification property.  
In Section~\ref{facet-ray identification for graphs with no K_5 minors}, we demonstrate that any graph without $K_5$ minor has the facet-ray identification property, thereby proving Theorem~\ref{main theorem graphs with no K_5 minors}.  
We then identify the ranks of the corresponding extremal rays.  
In Section~\ref{characterization of all extremal ranks for series-parallel graphs}, we prove Theorem~\ref{main theorem characterization of all extremal ranks}, showing that if $G$ is also series-parallel then the facets are enough to characterize all extremal rays of $\symG$.
Finally, in Section~\ref{graphs without the facet-ray identification property}, we discuss how to identify graphs that do not have the facet-ray identification property.


\section{Preliminaries.}
\label{preliminaries}


\subsection{Graphs.}
\label{graph preliminaries}
For a graph $G$ with vertex set $[p]$ and edge set $E$ we let $\nonedge$ denote the set of \emph{nonedges} of $G$, that is, all unordered pairs $\{i,j\}$ for which $i,j\in[p]$ but $\{i,j\}\notin E$.
Then we define the \emph{complement of $G$} to be the graph $G^c$ on the vertex set $[p]$ with edge set $\nonedge$.  
Recall that a subgraph of $G$ is any graph $H$ whose vertex set is a subset of $[p]$ and whose edge set is a subset of $E$.  
A subgraph $H$ of $G$ with edge set $E^\prime$ is called \emph{induced} if there exists a subset $V^\prime\subset[p]$ such that the vertex set of $H$ is $V^\prime$ and $E^\prime$ consists of all edges of $G$ connecting any two vertices of $V^\prime$.   
We let $K_p$ denote the complete graph on $p$ vertices, $C_p$ denote the cycle on $p$ vertices, and $K_{m,n}$ denote the complete bipartite graph where the vertex set is the disjoint union of $[m]$ and $[n]$.  
A subgraph $H$ of $G$ is called a \emph{chordless cycle} if $H$ is an induced cycle subgraph of $G$.  
A graph $G$ is called \emph{chordal} if every chordless cycle in $G$ has at most three edges.  
We can \emph{delete} an edge of $G$ by removing it from the edge set $E$, and \emph{contract} an edge $\{i,j\}$ of $G$ by identifying the two vertices $i$ and $j$ and deleting any multiple edges introduced by this identification.  
Similarly, we \emph{delete} a vertex of $G$ by removing it from the vertex set of $G$ as well as all edges of $G$ attached to it.  
A graph $H$ is called a \emph{minor} of $G$ if $H$ can be obtained from $G$ via a sequence of edge deletions, edge contractions, and vertex deletions.


\subsection{Sparsity order of $G$.}
\label{sparsity order preliminaries}
We are interested in $\symG$, the closed convex cone consisting of all $p\times p$ positive semidefinite matrices with zeros in the $ij^{th}$ entry for all $\{i,j\}\in\nonedge$.  
Recall that a matrix $X\in\symG$ is \emph{extremal} in $\symG$ if it lies on an extreme ray of $\symG$.  
The \emph{(sparsity) order} of $G$, denoted $\ord(G)$, is the maximum rank of an extremal matrix in $\symG$.  
In \cite{agler} the authors develop a general theory for studying graphs $G$ with high sparsity order.  
Fundamental to their theory is the so-called \emph{dimension theorem}, which is stated in terms of the expression of a positive semidefinite matrix as the Gram matrix for a collection of vectors.
Recall that a  (real) $p\times p$ matrix $X=[x_{ij}]$ is positive semidefinite if and only if there exist vectors $u_1,u_2,\ldots,u_p\in\R^k$ such that $x_{ij}=u_i^Tu_j$.  
The sequence of vectors $(u_1,\ldots,u_p)$  is called a ($k$-dimensional) \emph{Gram representation} of $X$, and if $X$ has rank $k$ this sequence of vectors is unique up to orthogonal transformation.  
Following the notation of \cite{laurent1}, for a subset $A\subset E\cup\nonedge$ define the set of $p\times p$ matrices 
	$$U_A:=\{u_iu_j^T+u_ju_i^T : \{i,j\}\in A\}.$$
The real span of $U_{\nonedge}$ is a subspace of the trace zero $k \times k$ real symmetric matrices that we call the \emph{frame space of $X$}.  
The following theorem proven in~\cite{agler} says that a matrix is extremal in $\symG$ if and only if this frame space is the entire trace zero subspace of $\s^k$.  

\begin{thm}\cite[Corollary 3.2]{agler}
\label{agler thm}
Let $X\in \symG$ with rank $k$ and $k$-dimensional Gram representation $(u_1,\ldots,u_p)$.  Then $X$ is extremal if and only if 
	$$\rank\left(U_{\nonedge}\right)={k+1\choose 2}-1.$$
\end{thm}

Furthermore, in \cite{agler} it is shown that $\ord(G) = 1$ if and only if $G$ is a chordal graph.  
Using Theorem~\ref{agler thm}, the authors then develop a general theory for detecting existence of higher rank extremals in $\symG$ based on a fundamental collection of graphs.  
A graph $G$ is called a \emph{$k$-block} provided that $G$ has order $k$ and no proper induced subgraph of $G$ has order $k$.  
The $k$-blocks are useful for identifying higher rank extremals since if $H$ is an induced subgraph of $G$ then $\ord(H)\leq\ord(G)$ \cite{agler}.  
In \cite{agler} it is also shown that the cycle on $p$ vertices is a $(p-2)$-block.  
A particularly important collection of $k$-blocks are the $k$-superblocks, the $k$-blocks with the maximum number of edges on a fixed vertex set.  
Formally, a \emph{$k$-superblock} is a $k$-block whose complement has precisely $(k^2+k-2)/2$ edges.  
Understanding the $k$-blocks and $k$-superblocks is equivalent to understanding their complements.  
In \cite[Theorem 1.5]{agler} the $3$-blocks are characterized in terms of their complement graphs, and in \cite[Theorem 0.2]{HLW} the $4$-superblocks are characterized in a similar fashion.  

In related works the structure of the graph $G$ is again used to describe the extreme ranks of $G$.  
In \cite{helton} it is shown that if $G$ is a clique sum of two graphs $G_1$ and $G_2$ then $\ord(G)=\max\{\ord(G_1),\ord(G_2)\}$, and $\ord(G) \leq p-2$ with equality if and only if $G$ is a $p$-cycle.  
Similarly, in \cite{grone}  the order of the complete bipartite graph $K_{p,m}$ is determined and it is shown that all ranks $1,2,\ldots,\ord(K_{p,m})$ are extremal.    


\subsection{The cut polytope of $G$.}
\label{cut polytope preliminaries}
First recall that to define the cut polytope in the variables $0$ and $1$ we assign to each cutset $\delta(U)$ a $(0,1)$-vector $x^{\delta(U)}\in\R^E$ with a $1$ in coordinate $x_e$ if and only if $e\in \delta(U)$.  
The polytope $\cut{G}$ is the convex hull of all such vectors, and it is affinely equivalent to $\cutpm{G}$ under the coordinate-wise transformation $x_e \mapsto 1-2x_e$ on $\R^E$.  
In order to prove that a graph $G$ has the facet-ray identification property we need an explicit description of the facet-supporting hyperplanes of $\cutpm{G}$, or equivalently, those of $\cut{G}$.  
For the complete graph $K_p$ one of the most interesting classes of valid inequalities for $\cutpm{G}$ are the \emph{hypermetric inequalities}.  
For an integer vector $b = (b_1,\ldots, b_p)$ satisfying $\sum_{i = 1}^p b_i = 1$ we call 
$$
\sum_{1\leq i < j \leq p}b_ib_j\leq \sum_{1\leq i < j\leq p}b_ib_jx_{ij}
$$ 
the \emph{hypermetric inequality} defined by $b$. 
Notice that every facet-supporting hypermetric inequality identifies an extreme ray in $\s^p_{\succeq0}=\s^p_{\succeq0}(K_p)$. However, despite the large collection of hypermetric inequalities, not all complete graphs have the facet-ray identification property. 
Moreover, since the only extreme rank of $\s^p_{\succeq0}$ is $1$, we are mainly interested in facet-defining inequalities that identify higher rank extreme rays for sparse graphs.  
The hypermetric inequalities generalize a collection of facet-defining inequalities of $\cutpm{G}$ called the \emph{triangle inequalities}, i.e. the hypermetric inequalities defined by $b = (1,1,-1)$.  
The triangle inequalities admit a second generalization to a collection of facet-defining inequalities for sparse graphs as follows: 
Let $C_m$ be a cycle in a graph $G$ and let $F\subset E(C_m)$ be an odd cardinality subset of the edges of $C_m$.   
The inequality 
$$
\sum_{e\in E(C_m)\backslash F}x_e - \sum_{e\in F}x_e \leq p-2
$$ 
is called a \emph{cycle inequality}.  
Using these inequalities Barahona and Mahjoub~cite{BM86} provide a linear description of $\cutpm{G}$ for any graph without $K_5$ minors.  
\begin{thm}
\label{barahona theorem}
\cite[Barahona and Mahjoub]{BM86}
Let $G$ be a graph with no $K_5$ minor.  
Then $\cutpm{G}$ is defined by the collection of hyperplanes
\begin{enumerate}[(1)]
	\item $-1\leq x_e\leq 1$ for all $e\in E(G)$, and 
	\item $\sum_{e\in E(C_m)\backslash F}x_e-\sum_{e\in F}x_e \leq m-2$ for all chordless cycles $C_m$ of $G$ and any odd cardinality subset $F\subset E(C_m)$.  
\end{enumerate}
\end{thm}

Suppose that $C_p$ is a chordless cycle in a graph $G$ without $K_5$ minors.  
For an odd cardinality subset $F\subset E(C_p)\subset E(G)$ define the vector $v^F\in \R^E$, where
$$
v^F_e=
\begin{cases} 
-1 & \mbox{ if $e\in F$}, \\
 1 & \mbox{ if $e\in E(C_m)\backslash F$}, \\
 0 & \mbox{ if $e\in E(G)\backslash E(C_m)$}. 
\end{cases}
$$
Similarly, let $v^e$ denote the standard basis vector for coordinate $e\in E(G)$ in $\R^E$. 
Then by Proposition~\ref{barahona theorem} we see that the facet-supporting hyperplanes of $\cutpm{G}$ are
\begin{enumerate}[(1)]
	\item $\langle \pm v^e, x \rangle=1$ for all $e\in E$, and
	\item $\langle v^F, x\rangle = m-2$ for all odd cardinality subsets $F\subset E(C_m)$ for all chordless cycles $C_m$ in $G$.  
\end{enumerate} 

In Section~\ref{facet-ray identification for graphs with no K_5 minors}, we identify for each facet-supporting hyperplane $\langle \alpha,x\rangle = b$ of $\cutpm{G}$ an extremal matrix $A=[a_{ij}]$ in $\symG$ of rank $b$ in which the off-diagonal nonzero entries $a_{ij}$ are given by the coordinates $\alpha_{e}$, for $e= \{i,j\}$, of the facet normal $\alpha\in\R^E$.  
In Section~\ref{characterization of all extremal ranks for series-parallel graphs}, we then show that the ranks $b$ are \emph{all} extremal ranks of $\symG$ as long as $G$ is also series-parallel.  
To do so, it will be helpful to have the following well-known and easy to prove lemma on the cut polytope of the cycle. 
\begin{lem}
\label{cut polytope of the cycle lemma}
The vertices of $\cutpm{C_p}$ are all $(\pm1)$-vectors in $\R^E$ containing an even number of $-1$'s. 
\end{lem}

The polytope $\cutpm{C_p}$ appears in the literature as the \emph{$p$-halfcube} or \emph{demihypercube}.


\section{The Geometry of Facet-Ray Identification}
\label{the geometric correspondence}

In this section, we examine the underlying geometry of the facet-ray identification property.  
Recall that the facet-ray identification property is defined to capture the following geometric picture.
Since $\cutpm{G}\subset\E_G$ 
then if all singular points on the boundary of $\E_G$ are also singular points on the boundary of $\cutpm{G}$, the supporting hyperplanes of facets of $\cutpm{G}$ will be translations of supporting hyperplanes of regular extreme points of $\E_G$ or facets of $\E_G$.  
It follows that the outward normal vectors to the facets of $\cutpm{G}$ generate the normal cones to these regular points and facets of $\E_G$.  
In the polar, the facet-normal vectors of $\cutpm{G}$ are then extreme points of $\E_G^\circ$, and consequently projections of extreme matrices of $\symG$.  
Thus, we can expect to find extremal matrices in $\symG$ whose off-diagonal entries are given by the facet-normal vectors of $\cutpm{G}$. 
Since the geometry of the elliptope is not at all generic this picture is, in general, difficult to describe from the perspective of real algebraic geometry. 
In Section~\ref{four cycle example} we provide this geometric picture in the case of the cycle on four vertices.  
This serves to demonstrate the difficultly of the algebro-geometric approach for an arbitrary graph $G$, and consequently motivate the combinatorial work done in the coming sections.
Following this example, we prove Theorem~\ref{dual body theorem}, the key to facet-ray identification.


\subsection{Geometry of the $4$-cycle: an example.}
	\label{four cycle example}
Consider the cycle on four vertices, $C_4$.  
For simplicity, we let $G:= C_4$, and we identify $\R^{E(G)}\simeq\R^4$ by identifying edge $\{i,i+1\}$ with coordinate $i$ for all $i=1,2,3,4$. Here we take the vertices of $C_4$ modulo 4.
By Lemma~\ref{cut polytope of the cycle lemma}, the cut polytope of $G$ is the convex hull of all $(\pm 1)$-vectors in $\R^4$ containing precisely an even number of $-1$'s.  
Equivalently, $\cutpm{G}$ is the $4$-cube $[-1,1]^4$ with truncations at the eight vertices containing an odd number of $-1$'s.  
Thus, $\cutpm{G}$ has sixteen facets supported by the hyperplanes
$$
\pm x_i =1, 
\qquad
\mbox{and}
\qquad 
\langle v^F, x \rangle = 2,
$$
where $F$ is an odd cardinality subset of $\{1,2,3,4\}$, and $v^F$ is the corresponding vertex of the 4-cube $[-1,1]^4$ with an odd number of $-1$'s.  

Proving that the 4-cycle $G$ has the facet-ray identification property amounts to identifying for each facet of $\cutpm{G}$ an extremal matrix in $\symG$  whose off-diagonal entries are given by the normal vector to the supporting hyperplane of the facet.  
For example, the facets supported by the hyperplanes $\pm x_1 = 1$ correspond to the rank $1$ extremal matrices
$$
\begin{array}{ccc}
Y = 
\begin{bmatrix}
1	&	1	&	0	&	0	\\
1	&	1	&	0	&	0	\\
0	&	0	&	0	&	0	\\
0	&	0	&	0	&	0	\\
\end{bmatrix}
& 
\mbox{and}
&

Y = 
\begin{bmatrix}
1	&	-1	&	0	&	0	\\
-1	&	1	&	0	&	0	\\
0	&	0	&	0	&	0	\\
0	&	0	&	0	&	0	\\
\end{bmatrix}.
\end{array}
$$
Similarly, the facets $\langle v^F,x\rangle = 2\,$ for $v^F = (1,-1,1,1)$ and $v^F = (1,-1,-1,-1)$ respectively correspond to the rank $2$ extremal matrices
$$
\begin{array}{ccc}
Y = \frac{1}{3}
\begin{bmatrix}
1	&	-1	&	0	&	-1	\\
-1	&	2	&	1	&	0	\\
0	&	1	&	1	&	-1	\\
-1	&	0	&	-1	&	2	\\
\end{bmatrix}
& 
\mbox{and}
&

Y = \frac{1}{3}
\begin{bmatrix}
1	&	-1	&	0	&	1	\\
-1	&	2	&	1	&	0	\\
0	&	1	&	1	&	1	\\
1	&	0	&	1	&	2	\\
\end{bmatrix}.
\end{array}
$$
As indicated by Theorem~\ref{dual body theorem}, these four matrices respectively project to four extreme points in $\E_G^\circ$, namely
$$
(1,0,0,0), \quad (-1,0,0,0), \qquad \frac{1}{3}(-1,1,-1,-1), \qquad \mbox{and} \qquad \frac{1}{3}(-1,1,1,1),
$$
with the former two being vertices of $\E_G^\circ$ (extreme points with full-dimensional normal cones) and the latter two being regular extreme points on the rank $2$ locus of $\E_G^\circ$.  
Indeed, all extreme points corresponding to the facets $\pm x_i =1$ will be rank $1$ vertices of $\E_G^\circ$, and all points corresponding to the facets $\langle v^F,x\rangle = 2$ will be rank $2$ regular extreme points of $\E_G^\circ$.  
Consequently, all sixteen points arise as projections of extremal matrices of $\symG$ of the corresponding ranks.  

To see why these sixteen points in $\E_G^\circ$ are extreme points of the specified type we examine the relaxation of $\cutpm{G}$ to $\E_G$, and the stratification by rank of the spectrahedral shadow $\E_G^\circ$.  
We compute the algebraic boundary of $\E_G$ as follows.  
The \emph{$4$-elliptope} is the set of correlation matrices 
$$
\E_4 = \left\{
(x_1,x_2,x_3,x_4,u,v)\in\R^6  :
X = \begin{pmatrix}
1	&	x_1	&	u	&	x_4	\\
x_1	&	1	&	x_2	&	v	\\
u	&	x_2	&	1	&	x_3	\\
x_4	&	v	&	x_3	&	1	\\
\end{pmatrix}
\succeq0
 \right\}.
$$
The algebraic boundary $\partial\E_4$ of $\E_4$ is defined by the vanishing of the determinant $D:=\det(X)$. The elliptope of the 4-cycle is defined as $\E_{G}:=\pi_{G}\left(\E_4\right)$ where
$$
\pi_{G}:\s^4\longrightarrow\R^{E(G)}, \qquad X_{i,i+1} \mapsto x_{i}.
$$
To identify the algebraic boundary of $\E_{G}$ we form the ideal 
$I:=\left(D,\frac{\partial}{\partial u}D,\frac{\partial}{\partial v}D\right),$
and eliminate the variables $u$ and $v$ to produce an ideal $J\subset\R[x_1,x_2,x_3,x_4]$.  
Using Macaulay2 we see that $J$ is generated by the following product of eight linear terms corresponding to the rank 3 locus,
$$
(x_1 - 1)(x_1 + 1)(x_2 - 1)(x_2 + 1)(x_3 - 1)(x_3 + 1)(x_4 - 1)(x_4 + 1),
$$
and the following sextic polynomial (with multiplicity two) corresponding to the rank 2 locus,
\begin{equation*}
\begin{split}
(&4x_1^2x_2^2x_3^2-4x_1^3x_2x_3x_4-4x_1x_2^3x_3x_4-4x_1x_2x_3^3x_4+4x_1^2x_2^2x_4^2
+4x_1^2x_3^2x_4^2+4x_2^2x_3^2x_4^2-4x_1x_2x_3x_4^3\\
&+x_1^4-2x_1^2x_2^2+x_2^4-2x_1^2x_3^2-2x_2^2x_3^2+x_3^4+8x_1x_2x_3x_4-2x_1^2x_4^2-2x_2^2x_4^2-2x_3^2x_4^2+x_4^4)^2.\\
\end{split}
\end{equation*}
The sextic factor is the \emph{$4^{th}$ cycle polynomial} $\Gamma_4^\prime$ as defined in \cite{sturmfels}. To visualize the portion of the elliptope cut out by this term we treat the variable $x_4$ as a parameter and vary it from 0 to 1.  
A few of these level curves (produced using Surfex) are presented in Figure \ref{fig: level curves}.
	\begin{figure}
	\centering
	$\begin{array}{c c c}
	\includegraphics[width=0.33\textwidth]{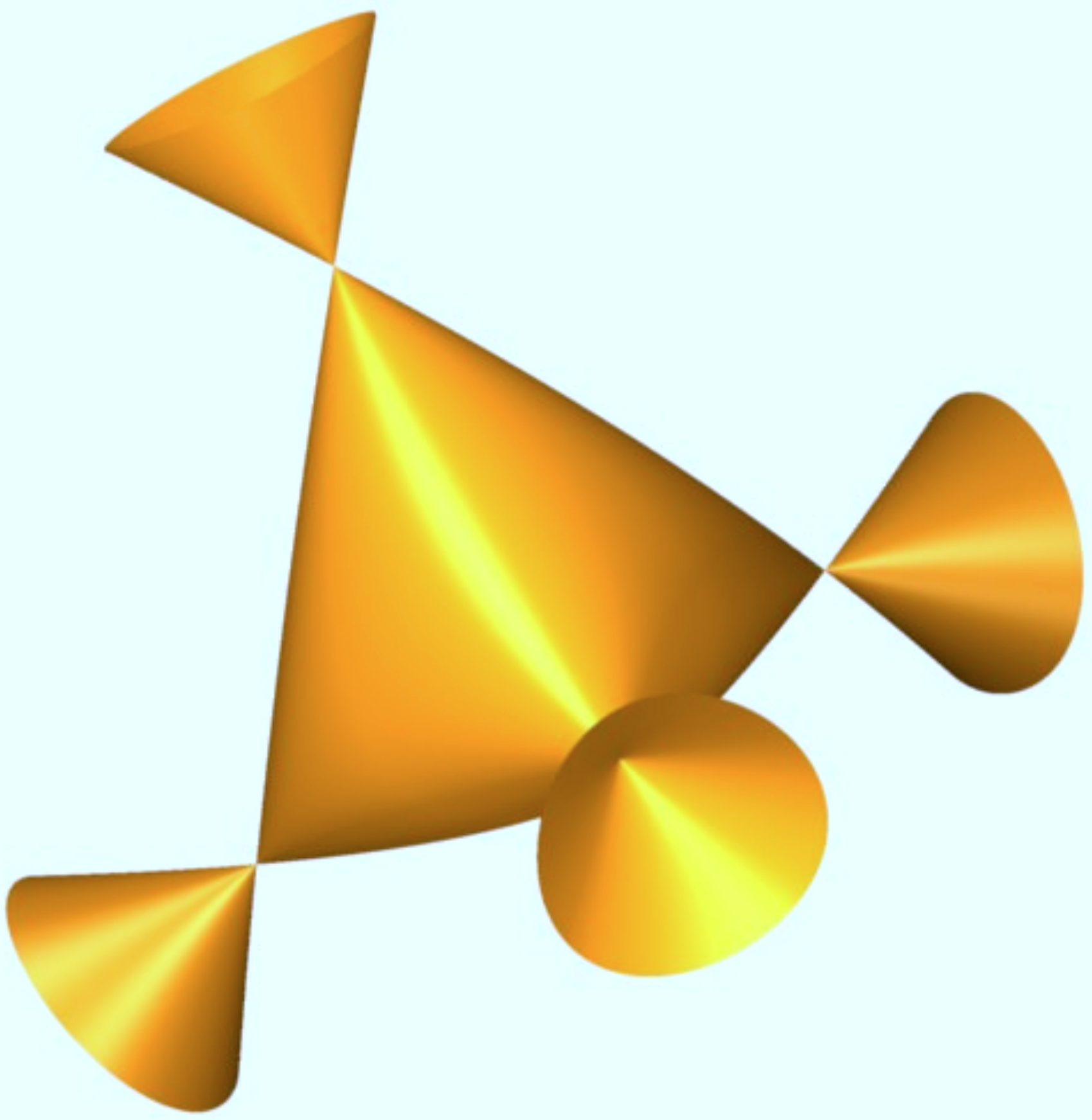} &
	\includegraphics[width=0.33\textwidth]{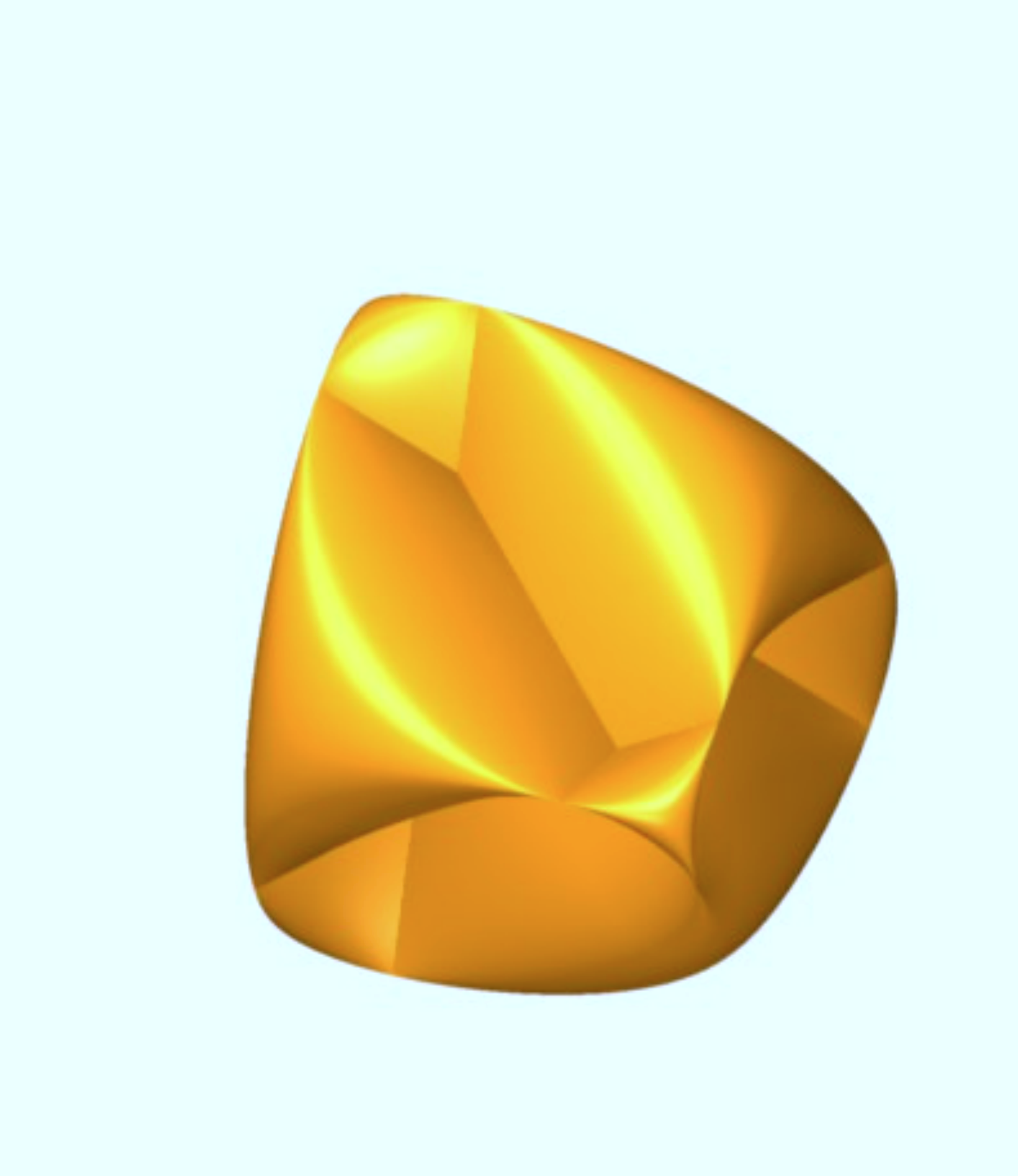} &
	\includegraphics[width=0.33\textwidth]{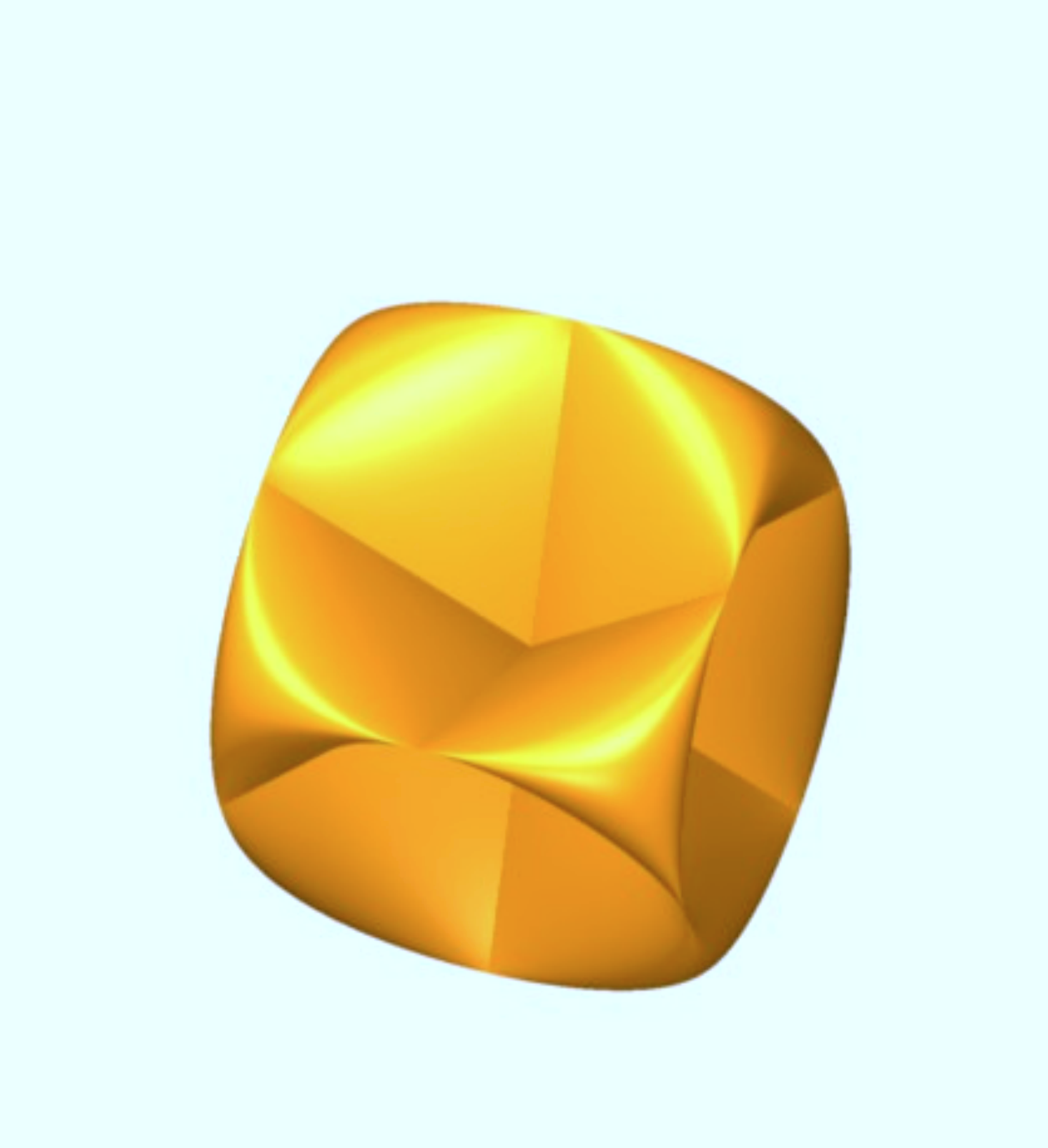} \\
	\end{array}$
	\caption{Level curves of the rank $2$ locus of $\E_{C_4}$.  The value of $x_4$ varies from $1$ to $0$ as we view the figures from left-to-right.}
	\label{fig: level curves}
	\end{figure}
An interesting observation is that the level curve with $x_4=1$ is the Cayley nodal cubic surface, the bounded region of which is precisely the elliptope $\E_{C_3}$.  We note that this holds more generally, i.e., the cut polytope of the $p$-cycle ${C_p}$ is the $p$-halfcube, and the facets of this polytope that lie in the hyperplane $\pm x_i=1$ are $(p-1)$-halfcubes. Thus, the elliptope $\E_G$ demonstrates the same recursive geometry exhibited by the polytope it relaxes.  

The eight linear terms define the rank $3$ locus as a hypersurface of degree eight.  
Since $\cutpm{G}\subset\E_G\subset[-1,1]^p$, the eight linear terms of the polynomial $p$ indicate that the facets of $\cutpm{G}$ supported by the hyperplanes $\pm x_i=1$ are also facets of $\E_G$.  
From this we can see that the eight hyperplanes $\pm x_i =1$ correspond to vertices in $\E_G^\circ$.
We can also see from this that only the simplicial facets of $\cutpm{G}$ have been relaxed in $\E_G$, and this relaxation is defined by the hypersurface $\{\Gamma_4^\prime=0\}$.  

Recall that we would like the relaxation of the facets to be smooth in the sense that all singular points on the boundary $\partial\E_G$ are also singular points on the boundary $\partial\cutpm{G}$. 
If this is the case, then we may translate the supporting hyperplanes of the relaxed facets to support regular extreme points of $\E_G$.  
The normal vectors to these translated hyperplanes will then form regular extreme points in the polar body $\E_G^\circ$.  
To see that this is indeed the case, we check that the intersection of the singular locus of $\{\Gamma_4^\prime=0\}$ with $\partial\E_G$ is restricted to the rank 3 locus of $\E_G$.  
With the help of Macaulay2, we compute that $\{\Gamma_4^\prime=0\}$ is singular along the six planes given by the vanishing of the ideals
$$
\begin{array}{ccc}
\langle x_3-x_4,x_1-x_2 \rangle,	&	 \langle x_3+x_4,x_1+x_2 \rangle,	 & 	\langle x_2-x_3,x_1-x_4 \rangle, \\
\langle x_2+x_3,x_1+x_4 \rangle,     &	 \langle x_2+x_4,x_1+x_3 \rangle, 	 & 	\langle x_2-x_4,x_1-x_3 \rangle,
\end{array}
$$
and at eight points
$$
\begin{array}{cccc}
\frac{1}{\sqrt{3}}\left(\mp1,\pm1,\pm1,\pm1\right), 	&	\frac{1}{\sqrt{3}} \left(\pm1,\mp1,\pm1,\pm1\right),	&
\frac{1}{\sqrt{3}}\left(\pm1,\pm1,\mp1,\pm1\right), 	&	 \frac{1}{\sqrt{3}}\left(\mp1,\mp1,\mp1,\pm1\right).
\end{array}
$$
The six planes intersect $\partial\E_G$ only along the edges of $\cutpm{G}$, and therefore do not introduce any new singular points that did not previously exist in $\cutpm{G}$.  
The eight singular points sit just outside the cut polytope $\cutpm{G}$ above the barycenter of each simplicial facet.  
However, these singular points lie in the interior of $\E_G$.  
This can be checked using the polyhedral description of $\E_G$ first studied by Barrett et al.~\cite{barrett}.  
The idea is that each point $(x_1,x_2,x_3,x_4)$ of the elliptope $\E_G$ arises from a point $(a_1,a_2,a_3,a_4)$ in the $(0,1)$-cut polytope, $\cut{G}$, by letting $x_i = cos(\pi a_i)$ for every $i\in[4]$.  
Since $\cut{G}$ is affinely equivalent to $\cutpm{G}$ under the linear transformation $y_i = 1-2x_i$, we apply the arccosine transformation of Barrett et al.~to the barycenter of each simplicial facet of $\cut{G}$ to produce the eight points on $\partial\E_G$:
$$
\begin{array}{cccc}
\frac{1}{\sqrt{2}}\left(\mp1,\pm1,\pm1,\pm1\right), 	&	\frac{1}{\sqrt{2}} \left(\pm1,\mp1,\pm1,\pm1\right),	&
\frac{1}{\sqrt{2}}\left(\pm1,\pm1,\mp1,\pm1\right), 	&	 \frac{1}{\sqrt{2}}\left(\mp1,\mp1,\mp1,\pm1\right).
\end{array}
$$
Thus, each of the eight singular points of $\Gamma_4^\prime$ lies in the interior of $\E_G$ on the line between the barycenter of a simplicial facet of $\cutpm{G}$ and one of these eight points in $\partial\E_G$.  
From this we see that the relaxation of the simplicial facets of $\cutpm{G}$ is smooth, and so we may translate the supporting hyperplanes $\langle v^F,x \rangle = 2$ away from $\cutpm{G}$ until they support some regular extreme point on $\partial\E_G$.  

In the polar $\E_G^\circ$ we check that the normal vectors to the hyperplanes $\pm x_i=1$ form vertices of rank~$1$, and the normal vectors corresponding to the translated versions of the hyperplanes $\langle v^F,x\rangle = 2$ are regular points on the rank $2$ strata of $\E_G^\circ$.  
The polar $\E_G^\circ$ is the spectrahedral shadow
$$
\E_G^\circ = \left\{
(x_1,x_2,x_3,x_4)\in\R^4  : \exists a,b,c\in\R :
Y = \begin{pmatrix}
a	&	x_1	&	0	&	x_4		\\
x_1	&	b	&	x_2	&	0		\\
0	&	x_2	&	c	&	x_3		\\
x_4	&	0	&	x_3	&	2-a-b-c	\\
\end{pmatrix}
\succeq0
 \right\},
$$
and the matrix $Y$ is a trace two matrix living in the cone $\symG$.   
The rank $3$ locus of $\E_G^\circ$ can be computed by forming the ideal generated by the determinant of $Y$ and its partials with respect to $a,b,$ and $c$, and then eliminating the variables $a,b,$ and $c$ from the saturation of this ideal with respect to the $3\times 3$ minors of $Y$.  
The result is a degree eight hypersurface that factors into eight linear forms:
$$
\begin{array}{cccc}
\!\!x_1-x_2-x_3+x_4+1,\!	&	\!-x_1+x_2-x_3+x_4+1,\!	&	\!-x_1-x_2+x_3+x_4+1,\!	&	\!x_1+x_2+x_3+x_4+1,\!\!	\\
\!\!x_1-x_2-x_3+x_4-1,\!	& \!-x_1+x_2-x_3+x_4-1,\! 	&	\!-x_1-x_2+x_3+x_4-1,\!	&	\!x_1+x_2+x_3+x_4-1.\!\!		
\end{array}
$$
The eight points in $\E_G^\circ$ that are dual to the hyperplanes $\pm x_i=1$ are vertices of the convex polytope whose $H$-representation is given by these linear forms.  
These vertices are projections of rank $1$ matrices in $\symG$.  
Our remaining eight hyperplanes supporting regular extreme points in $\E_G$ should correspond to rank $2$ regular extreme points in $\E_G^\circ$.
We check that the normal vectors to these hyperplanes don't lie on the singular locus of the rank $2$ strata of $\E_G^\circ$.    
To compute the rank 2 strata of $\E_G^\circ$ we eliminate the variables $a,b,$ and $c$ from the ideal generated by the $3\times 3$ minors of $Y$ and all of their partial derivatives with respect to the variables $a,b,$ and $c$.  
The result is a degree four hypersurface defined by the polynomial
$$
x_1^2x_2^2x_3^2-x_1^3x_2x_3x_4-x_1x_2^3x_3x_4-x_1x_2x_3^3x
     _4+x_1^2x_2^2x_4^2+x_1^2x_3^2x_4^2+x_2^2x_3^2x_4^2-x_1x_2x_3x_4^3+
     x_1x_2x_3x_4
$$
that is singular along six planes defined by the vanishing of the ideals
$$
\begin{array}{cccccc}
	\langle u, y\rangle,	&	\langle u, z\rangle, 	&	\langle x, u \rangle,	&
	\langle z, y\rangle,	&	\langle x,z\rangle,	&	\langle x,y\rangle.	
\end{array}
$$
To visualize the rank $2$ locus of $\E_G^\circ$ we intersect this degree four hypersurface with the hyperplane $x_1+x_2+x_3+x_4=b$ and let $b$ vary from $0$ to $1$.  
A sample of these level curves is presented in Figure~\ref{fig: level curves of dual}.
	\begin{figure}
	\centering
	$\begin{array}{c c c}
	\includegraphics[width=0.28\textwidth]{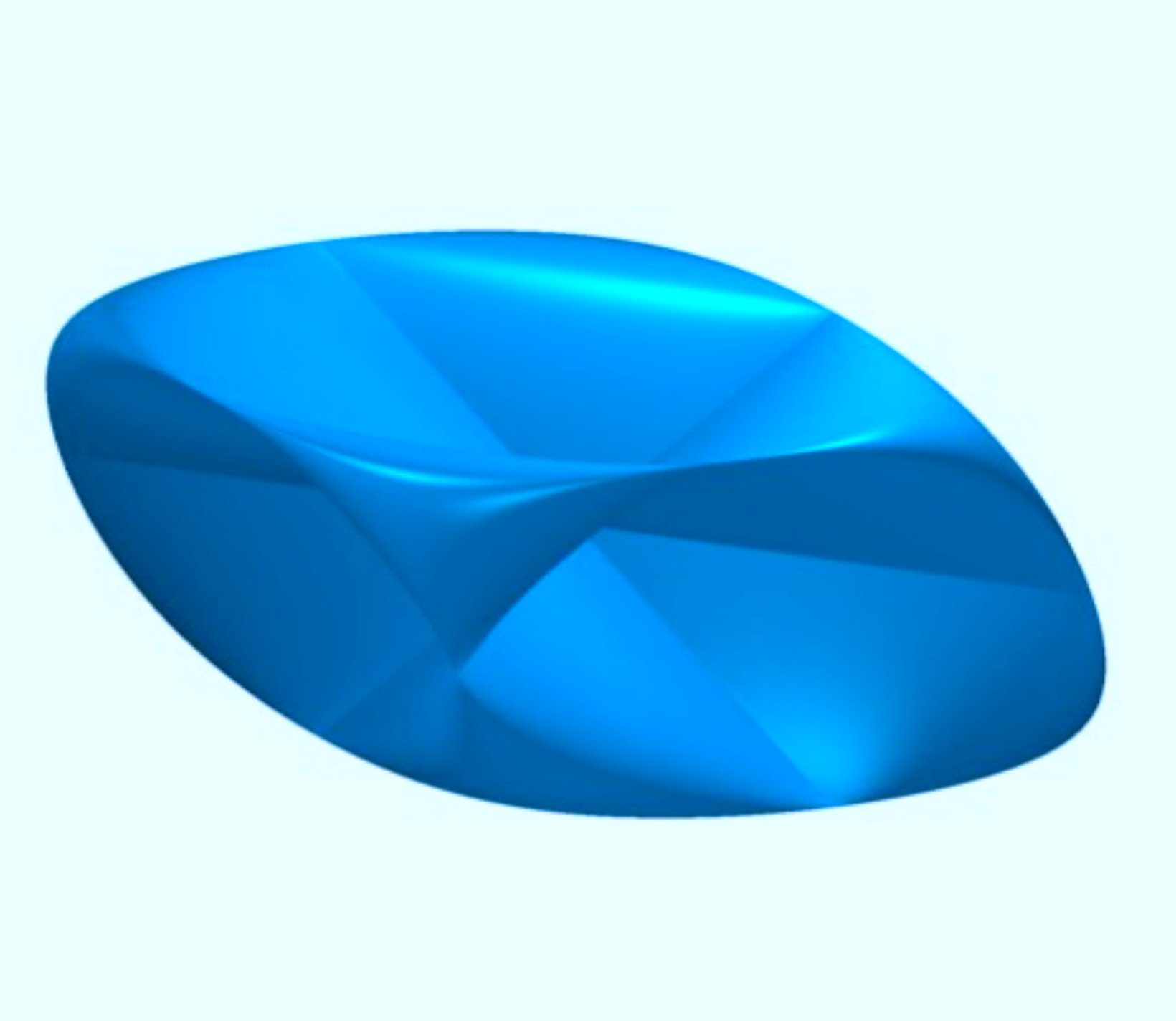}\quad &
\quad	\includegraphics[width=0.28\textwidth]{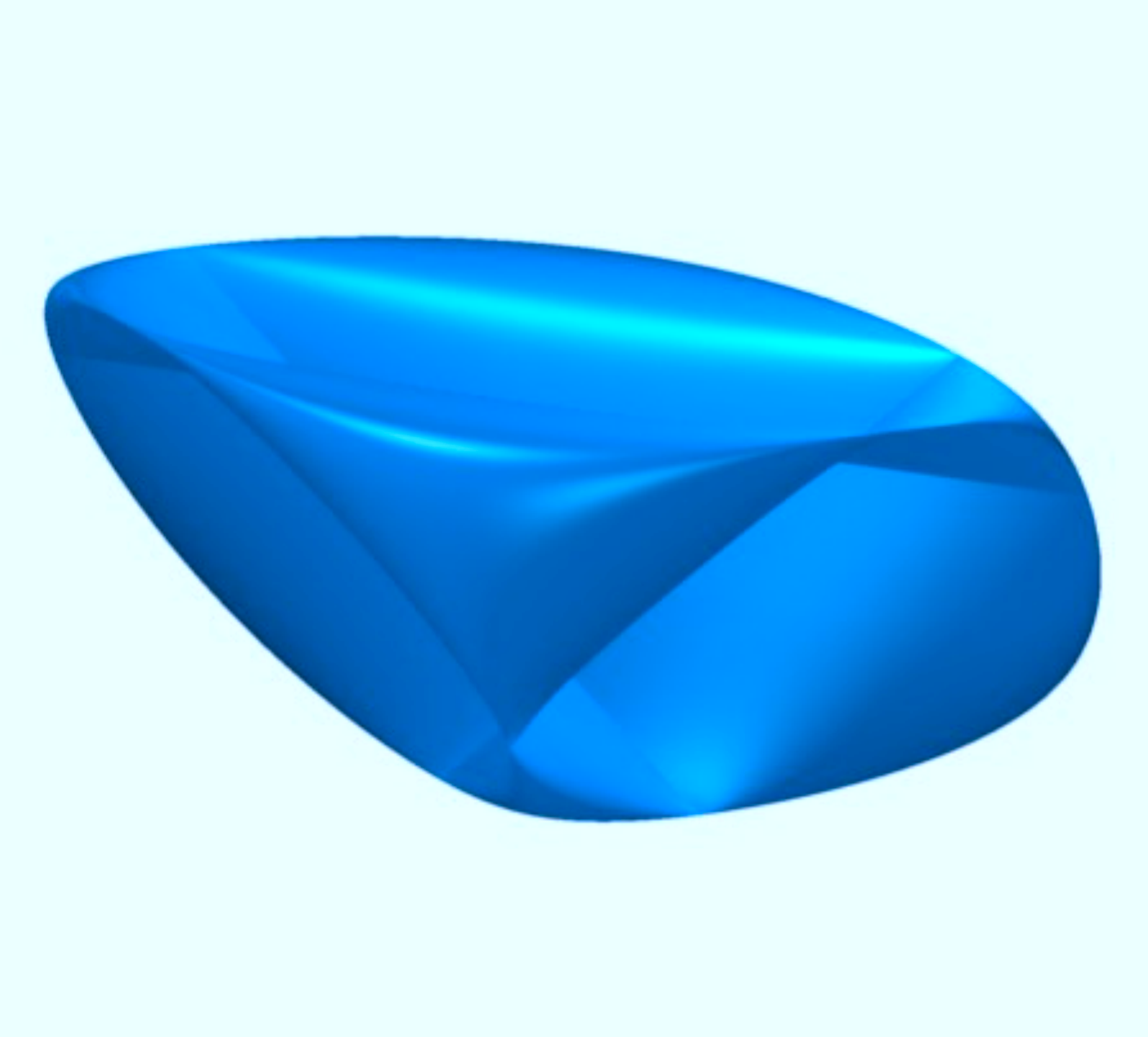}\qquad & \qquad
	\includegraphics[width=0.28\textwidth]{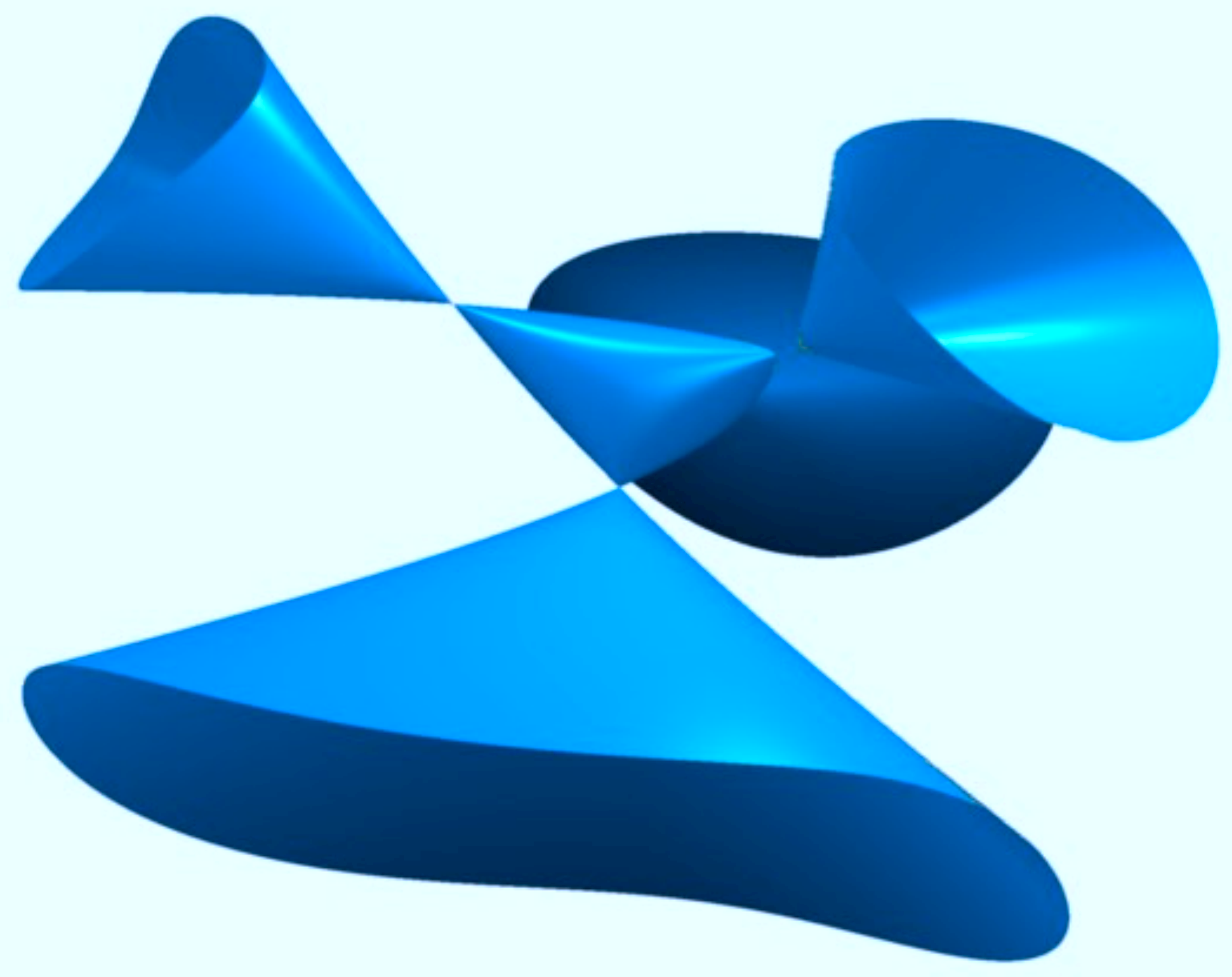} \\
	\end{array}$
	\caption{Level curves of the rank $2$ locus of $\E_{C_4}^\circ$ defined by the hyperplane $x_1+x_2+x_3+x_4=b$.  The value of $b$ varies from $1$ to $0$  from left-to-right.}
	\label{fig: level curves of dual}
	\end{figure}
Since the normal vectors to our hyperplanes are nonzero in all coordinates, their corresponding points are regular points in the rank $2$ locus 
of $\E_G^\circ$, and therefore arise as projections of extremal matrices of rank $2$ in $\symG$.  
The combinatorial work in Section~\ref{facet-ray identification for graphs with no K_5 minors} supports this geometry.


\subsection{The polar of an elliptope}
\label{the dual of an elliptope}
Recall that the polar of a subset $K\subset \R^d$ is 
\begin{equation}
\label{def:polar}
K^\circ = 
\{
y\in\R^d : \langle x,y\rangle\leq 1 \mbox{ for all $x\in K$}
\}.
\end{equation}
In this subsection we prove Theorem~\ref{dual body theorem} via an application of spectrahedral polarity.  
We first review how to compute the polar for a spectrahedron via the methods of Ramana and Goldman described in \cite{ramana}.

Let $C, A_1,\dots , A_d\in \s^p$, where $A_1,\dots , A_d$ are linearly independent. 
A \emph{spectrahedron} is a closed convex set $\mathcal{S}$ of the form
$$
\mathcal{S} = 
\left\{
x\in\R^d :  A(x)=C+\sum_{i=1}^dx_iA_i\succeq 0
\right\},
$$
where $A(x)\succeq0$ indicates that $A(x)$ is positive semidefinite.  
Since the matrices $A_1,\ldots,A_d$ are linearly independent then $\mathcal{S}$ is affinely equivalent to the section of the positive semidefinite cone
$$
\mathcal{A}\cap\s^p_{\succeq0},
$$
where
$\mathcal{A}=C+\textrm{span}_\R(A_1,\dots , A_d)$.  
Thus, the affine section $\mathcal{A}\cap\s^p_{\succeq0}$ is often also called a spectrahedron.
Let $\mathcal{W}=\textrm{span}_\R(A_1,\dots , A_d)$ be the linear subspace defined by $A_1,\dots , A_d$ and 
$$ \pi_\mathcal{W} \,: \,\, \s^p \rightarrow \s^p/{\mathcal W}^\perp \,\simeq \, \mathbb{R}^d, \qquad X\mapsto (\langle X,A_1\rangle, \dots , \langle X,A_d\rangle)$$
be the canonical projection. 
We define the ${p+1\choose 2}$-dimensional spectrahedron
$$
\mathcal{R}=\{X\in \s^p_{\succeq 0} : \langle X,C\rangle\leq 1\}.
$$
Then the polar of the spectrahedron $\mathcal{S}$ is a \emph{spectrahedral shadow}, namely the closure of the image of the spectrahedron $\mathcal{R}$ under the projection $\pi_\mathcal{W}$, i.e.
$\mathcal{S}^\circ=\textrm{cl}(\pi_\mathcal{W}(\mathcal{R}))$  \cite{ramana}.


\subsubsection*{Proof of Theorem~\ref{dual body theorem}}

We first apply the general theory about spectrahedra to compute the polar of the set of correlation matrices 
$$
\mathcal{E}_p=\{ X\in \s^p_{\succeq 0} :  X_{ii}=1 \textrm{ for all } i\in[p] \}.
$$
Let $A_{ij}=[a_{ij}]\in \s^p$, $1\leq i<j\leq p$ be the zero matrix except for $a_{ij}=a_{ji}=1$. 
Then $\mathcal{E}_p$ is a spectrahedron
$$
\mathcal{E}_p= \s^p_{\succeq 0}\cap \mathcal{A},
$$
 where $\mathcal{A}$ is the affine subspace
$$
\mathcal{A}= \textrm{I}_p+\textrm{span}_\R( A_{ij} : 1\leq i<j\leq p).
$$
Notice that since $C = \textrm{I}_p$, then $\mathcal{R}=\{X\in \s^p_{\succeq 0} : \langle X,C\rangle= 1\}$.  
Applying the above techniques we get that the polar of $\mathcal{E}_p$ is the spectrahedral shadow
$$
\mathcal{E}_p^\circ = \{ Y\in\mathbb{R}^{\binom{p}{2}} : \exists \,X\in  \s^p_{\succeq 0} \textrm{ such that } X_{ij}=Y_{ij} \textrm{ for all } i< j\,\textrm{ and } \, \textrm{tr}(X)= 2\}.
$$
We now compute the polar of the elliptope 
$$
\mathcal{E}_G=\{y \in \R^E : \exists \, Y\in \mathcal{E}_p \textrm{ such that } Y_E=y \}.
$$
Let $\mathcal{L}$ be a linear subspace of $\mathbb{R}^{\binom{p}{2}}$ defined by $A_{ij}$, $ij\in E$. 
We denote by $\mathcal{L}^\perp$ the orthogonal complement of $\mathcal{L}$ in $\mathbb{R}^{\binom{p}{2}}$. 
Then 
$$
 ((\mathcal{E}_p\,+\, \mathcal{L}^\perp)/\mathcal{L}^\perp)^\circ  \,\,\, = \,\,\,  \mathcal{E}_p^\circ \,\cap \, \mathcal{L}, 
$$
which means that
$$
\E_G^\circ = \{ x\in\mathbb{R}^E : \exists \,X\in \symG \textrm{ such that } X_E=x \,\textrm{ and } \, \textrm{tr}(X)= 2\}.
$$
This completes the proof of Theorem~\ref{dual body theorem}. \hfill$\square$

It is clear that the constraint \,$\textrm{tr}(X)= 2$\, is just a scaling. 
So the extreme points of the convex body $\E_G^\circ$ correspond to the extremal rays of $\symG$.
Since an extreme point of $\E_G^\circ$ either has a full-dimensional normal cone or is a regular point of $\E_G^\circ$ we arrive at the following corollary.  

\begin{cor}
\label{hyperplanes in the elliptope}
The hyperplanes supporting facets of the elliptope $\mathcal{E}_G$ or regular extreme points of $\E_G$ correspond to extremal rays of the cone $\,\symG$.
\end{cor}

A supporting hyperplane of $\E_G$ of the type described in Corollary~\ref{hyperplanes in the elliptope} identifies an extremal ray of rank $r$ if it corresponds to a point in the rank $r$ strata of $\E_G^\circ$.  
This is the basis for the facet-ray identification property.


\section{Facet-Ray Identification for graphs without $K_5$ minors}
\label{facet-ray identification for graphs with no K_5 minors}

In this section, we show that all graphs without $K_5$ minors have the facet-ray identification property.  
We first demonstrate that the $p$-cycle $C_p$ has the facet-ray identification property, and then generalize this result to all graphs without $K_5$ minors.


\subsection{Facet-Ray Identification for the Cycle.}
\label{facet-ray identification for the cycle}

Let $G:= C_p$ for $p\geq 3$.  
Here, we will make the identification $\R^E\simeq\R^p$ by identifying the coordinate $e=\{i,i+1\}$ in $\R^E$ with the coordinate $i$ in $\R^p$. 
For an edge $e\in E$ we define two $p\times p$ matrices, $X_e$ and $X_e^-$, where
\begin{center}
$\begin{array}{lcr}
	$$
	(X_e)_{s,t}:=
	\begin{cases}
	 1 & \mbox{ if $s,t\in e$}, \\
	 0 & \mbox{ otherwise}, 
	\end{cases}
	$$
& 
	\quad\,\mbox{and} \quad\quad
&
	$$
	(X_e^-)_{s,t}:=
	\begin{cases}
	 1 & \mbox{ if $s=t$ and $s,t\in e$}, \\ 
	-1 & \mbox{ if $s\neq t$ and $s,t\in e$}, \\
	 0 & \mbox{ otherwise}. 
	 \end{cases}
	 $$
\end{array}$
\end{center}

	\begin{prop}
	\label{halfcubical matrices}
	The matrices $X_e$ and $X_e^-$ are extremal in $\symG$ of rank $1$.  
	Moreover, the off-diagonal entries of $X_e$ and $X_e^-$ are given by the normal vector to the hyperplane $\langle \pm v^e, x \rangle=1$, respectively.  
	\end{prop}
	
\begin{proof}
These matrices are of rank $1$ and have respective $1$-dimensional Gram representations $(u_1^e,\ldots, u_p^e)$ and $(w_1^e,\ldots, w_p^e)$, where
\begin{center}
$\begin{array}{lcr}
	$$u^e_{t}:=\begin{cases} 1 & \mbox{ if $t\in\{i,j\}$}, \\ 0 & \mbox{ otherwise}, \end{cases}$$ & \mbox{and} &
	$$w^e_{t}:=\begin{cases} 1 & \mbox{ if $t=i$}, \\ 
		-1 & \mbox{ if $t=j$}, \\ 0 & \mbox{ otherwise}. \end{cases}$$
\end{array}$
\end{center}
Consider the collection $U_{\nonedge}$ with respect to these Gram representations.  
If $\{s,t\}\in\nonedge$ then either $s\notin\{i,j\}$ or $t\notin\{i,j\}$ (or both).  
Thus, $u_s={\bf \overline{0}}$ ($w_s={\bf \overline{0}}$) or $u_t={\bf \overline{0}}$ ($w_t={\bf \overline{0}}$) (or both).  
Hence, $U_{\nonedge}=\{{\bf \overline{0}}\}$ and $\rank\left(U_{\nonedge}\right)=0$.  
So by Theorem \ref{agler thm} the matrices $X_e$ and $X_e^-$ are extremal in $\symG$.  
Since for all $e\not=e^\prime\in E$ the matrices $X_e$, $X_{e^\prime}$, $X_e^-$ and $X_{e^\prime}^-$ are not scalar multiples of each other, each such matrix lies on a different extremal ray of $\symG$.  
\end{proof}

Our next goal is to identify rank $p-2$ extremal matrices in $\symG$ whose off-diagonal entries are determined by the normal vectors to the facet-supporting hyperplanes $\langle v^F, x\rangle = p-2$.
Thus, we wish to prove the following theorem.

\begin{thm}
\label{simplicial facets correspondence thm}
Let $F\subset[p]$ be a subset of odd cardinality.  
Then there exists a rank $p-2$ extremal matrix $\Delta_{p,F}\in\symG$ such that $(\Delta_{p,F})_{i,i+1}=-v^F_i$ for all $i\in[p]$ (modulo $p$).  
\end{thm}

To prove Theorem \ref{simplicial facets correspondence thm} we first construct the matrices $\Delta_{p,F}$ when $F$ is a maximal odd cardinality subset of $[p]$, and then prove a lemma showing the existence of such matrices in all the remaining cases.  
Let $F$ be a maximal odd cardinality subset of $[p]$, and let $e_1,e_2,\ldots,e_{p-2}$ denote the standard basis vectors for $\R^{p-2}$.  
Notice that for $p$ even, $F=[p]\backslash\{i\}$ for some $i\in[p]$, and for $p$ odd, $F=[p]$.  
For $i\in[p]$ we define the collection of vectors

\begin{center}
$
\begin{array}{l}
u_i:=e_{p-2}, \\
u_{i+1}:=\sum_{j=1}^{p-2}(-1)^{j+1}e_j, \\
u_{i+2}:=e_1, \\
u_{j}:= e_{j-i-2}+e_{j-i-1},
\;\mbox{ for $\,i+3\leq j\leq i+p-1$}.\\
\end{array}
$
\end{center}


Here, we view the indices of these vectors modulo $p$, i.e. $p+1=1$.
For $p$ even and $F=[p]\backslash\{i\}$, let $\Delta_{p,F}$ denote the positive semidefinite matrix with Gram representation $(u_1,u_2,\ldots,u_p)$.  
Similarly, for $p$ odd and $F=[p]$, let $i=1$, and let $\Delta_{p,F}$ denote the matrix with Gram representation $(u_1,u_2,\ldots,u_p)$.

\begin{rmk}
\label{gram representation remark}
While independently discovered by the authors in terms of facets of $\cutpm{G}$, the Gram representation $(u_1,\ldots,u_{p})$ for $i=p-1$  was previously used in \cite[Lemma 6.3]{agler} to demonstrate that the sparsity order of the $p$-cycle is larger than $1$ for $p\geq4$.  
Here, we verify that this representation is indeed extremal, and show that it arises as part of a collection of extremal representations given by the facets of the cut polytope $\cutpm{G}$.  
\end{rmk}

\begin{lem}
\label{maximal cardinality T case}
Let $F$ be a maximal odd cardinality subset of $[p]$.  
Then the matrix $\Delta_{p,F}$ is extremal in $\symG$ with rank $p-2$.
\end{lem}

\begin{proof}
It is easy to check that all entries of $\Delta_{p,F}$ corresponding to nonedges of $G$ will contain a zero.  
Notice also that all adjacent pairs $u_j,u_{j+1}$  have inner product $1$ except for the pair $u_i,u_{i+1}$, when $p$ is even, whose inner product is $-1$.  
Moreover, $(u_1,u_2,\ldots,u_p)$ spans $\R^{p-2}$, and therefore $\rank(\Delta_{p,F})=p-2$.  
Thus, by Theorem \ref{agler thm}, it only remains to verify that $\rank(U_{\nonedge})={p-1\choose 2}-1$.  
However, since $\#\nonedge={p\choose 2}-p={p-1\choose 2}-1$, it suffices to show that the collection of matrices $U_{\nonedge}$ are a linearly independent set.  

Without loss of generality, we set $i=1$.
First it is noted that the vectors $u_3, \, u_4, \, \ldots u_p$ are
linearly independent in $\R^{p-2}$ and we consider
them as a basis of the vector space $\R^{(p-2)}$.  Thus we can write $u_1$ and $u_2$
as follows: 
\[
 \begin{array}{ccl}
u_1 & = & \sum_{i = 3}^p (-1)^{(p+i)} u_i,\\
u_2 & = & \sum_{i = 3}^p (-1)^i (i - 2) u_i.\\
\end{array}
\]

Since the graph $G$ is a cycle of length $p$, $\nonedge$ does not
contain $\{i, i+1\}$ for $i = 1, \ldots p-1$ and $\{p, 1\}$.  Thus, we
consider the set of matrices  
\[
V:= \{u_i \cdot u_j^T + u_j \cdot u_i^T : i = 3, \ldots, p, \, i < j, \, j
\not = i+1\} \subset \R^{(p-2)\times (p-2)}.
\]
Note that $u_i \cdot u_j^T + u_j \cdot u_i^T$ is a $(p-2) \times
(p-2)$ matrix $M$ whose $(i' j')^{th}$ element is
\[
M_{i' j'}=  \begin{cases}
1 & \mbox{if } i' = i, \, j' = j, \, j' \not = i' + 1,\\
1 & \mbox{if } i' = j, \, j' = i, \, j' \not = i' + 1,\\
0 & \mbox{otherwise}.\\
\end{cases} 
\]
Hence, the set
\[
V= \{u_i \cdot u_j^T + u_j \cdot u_i^T| i = 3, \ldots p, \, i < j, \,  j
\not = i+1\}
\]
is linearly independent.  

Now we consider the matrix $u_1 \cdot u_k^T + u_k\cdot u_1^T$ for $k =
3, \ldots , p-1$.  Note that
\[
\begin{array}{ccl}
u_1 \cdot u_k^T + u_k\cdot u_1^T &=& \left( \sum_{i = 3}^p
  (-1)^{(p+i)} u_i \right) \cdot u_k^T + u_k\cdot \left( \sum_{i = 3}^p
  (-1)^{(p+i)} u_i \right)^T\\
&=& \left( \sum_{i = 3}^p  (-1)^{(p+i)} u_i \cdot u_k^T \right) +
\left( \sum_{i = 3}^p
  (-1)^{(p+i)} u_k \cdot u_i^T \right)\\
& =: & \bar{M}^k,\\
\end{array}
\]
where 
\[
\bar{M}_{i' j'}^k = \begin{cases}
(-1)^{(p+i')} \cdot 2 & \mbox{if } i' = j' = k,\\
(-1)^{(p+i')}  & \mbox{if } i' \not = k, \, i' = 3, \ldots , (p - 1)\mbox{ and } j' = k,\\
(-1)^{(p+j')}  & \mbox{if } i' = k, \, j' \not = k,  \mbox{ and } j'=3, \ldots , (p - 1), \\
0 & \mbox{else.}
\end{cases}
\]
In addition, we consider the matrix $u_2 \cdot u_k^T + u_k\cdot u_2^T$ for $k =
4, \ldots , p$.  Note that
\[
\begin{array}{ccl}
u_2 \cdot u_k^T + u_k\cdot u_2^T &=& \left( \sum_{i = 3}^p (-1)^i (i -
  2) u_i \right) \cdot u_k^T + u_k\cdot \left( \sum_{i = 3}^p (-1)^i (i - 2) u_i \right)^T\\
&=& \left( \sum_{i = 3}^p (-1)^i (i - 2) u_i\cdot u_k^T \right) +
\left( \sum_{i = 3}^p
 (-1)^i (i - 2)  u_k \cdot u_i^T \right)\\
& =: & \tilde{M}^k,\\
\end{array}
\]
where 
\[
\tilde{M}_{i' j'}^k = \begin{cases}
(-1)^{i'} \cdot 2 \cdot (i' - 2)  & \mbox{if } i' = j' = k,\\
(-1)^{i'} \cdot  (i' - 2)  & \mbox{if } i' \not = k, \, i' = 3, \ldots , (p - 1)\mbox{ and } j' = k,\\
(-1)^{j'} \cdot (j' - 2) & \mbox{if } i' = k, \, j' \not = k,  \mbox{ and } j'=3, \ldots , (p - 1), \\
0 & \mbox{else.}
\end{cases}
\]
Since $V$ does not contain the matrices $\hat{M}^i$ for $i = 3, \ldots p$ such that 
\[
\hat{M}_{i' j'}^i=  \begin{cases}
1 & \mbox{if } i' = i, \, j' = i + 1,\\
0 & \mbox{otherwise,}\\
\end{cases} 
\]
and the matrices
$M'^i$ for $i = 3, \ldots p$ such that 
\[
{M'}_{i' j'}^i=  \begin{cases}
1 & \mbox{if } i' = i, \, j' = i,\\
0 & \mbox{otherwise, }\\
\end{cases} 
\]
we cannot write $\tilde{M}^k$ in terms of $\bar{M}^{k'}$ and
elements of $V$ (and also we cannot write $\bar{M}^k$ in terms of $\tilde{M}^{k'}$ and
elements of $V$) for $k = 3, \ldots, p-1$ and $k' = 4, \ldots, p$. Hence, the
matrices $\tilde{M}^k$ for $k = 4, \ldots, p$, $\bar{M}^k$ for $k' = 3,
\ldots, p-1$, and the matrices in $V$  are linearly independent.  
\end{proof}

To provide some intuition as to the construction of the remaining extremal matrices we note that a $k$-dimensional Gram representation of a graph $G$ with vertex set $[p]$ is a map $Y:[p]\longrightarrow \R^k$ such that $\spn_\R\{Y(i)|i\in[p]\}=\R^k$ and $Y(i)^TY(j)=0$ for all $\{i,j\}\in\nonedge$.
Hence, the Gram representation $(u_1,u_2,\ldots,u_p)$ is an inclusion of the graph $G$ into the hypercube $[-1,1]^{p-2}$.  
Here, the vertex $i$ of $G$ is identified with the vector $u_i\in\R^k$.  
In this way, the underlying cut $U$ of a cutset $\delta(U)$ of $G$ is now a collection of vectors as opposed to a collection of indices.  
We now consider the cutsets $\delta(U)$ of $G$ with respect to the representation $(u_1,u_2,\ldots, u_p)$ for the maximal odd cardinality subsets $F$, and negate the vectors in the underlying cut $U$ to produce the desired extremal matrices for lower cardinality odd subsets of $[p]$.
This is the content of the following lemma.

\begin{lem}
\label{lower cardinality T case}
Let $F\subset[p]$ be a subset of odd cardinality. 
There exists a rank $p-2$ extremal matrix $\Delta_{p,F}$ in $\symG$ with off-diagonal entries satisfying
	$$(\Delta_{p,F})_{s,t}=\begin{cases} 0 & \mbox{ if $\{s,t\}\in\nonedge$}, \\ 1 & \mbox{ if $t=s+1$ and $s\in F$}, \\ -1 & \mbox{ if $t=s+1$ and $s\notin F$}. \end{cases}$$
\end{lem}

\begin{proof}
We produce the desired matrices in two separate cases, when $p$ is odd and when $p$ is even.
Suppose first that $p$ is odd, and consider the $(p-2)$-dimensional Gram representation $(u_1,u_2,\ldots,u_p)$ defined above for the extremal matrix $\Delta_{p,[p]}$.  
This Gram representation includes $G$ into the hypercube $[-1,1]^{p-2}$ such that vertex $i$ of $G$ corresponds to $u_i$. 

We now consider the cuts of $G$ with respect to this inclusion.  
Recall from Section~\ref{preliminaries} that even subsets of $E$ are the cutsets $\delta(U)$ of $G$, and they correspond to a unique cut $(U,U^c)$ of $G$.  
For each $i\in[p]$ we can consider the edge $\{i,i+1\}\in E$.  
Let $F\subset[p]$ be of odd cardinality.  
Then $F^c$ is of even cardinality and hence has an associated cut $(U,U^c)$ such that $F^c=\delta(U)$.  
Now, thinking of $U\subset[p]=V(G)$, negate all vectors in $(u_1,u_2,\ldots,u_p)$ with indices in $U$ to produce a new $(p-2)$-dimensional representation of $G$, say $(w_1,w_2,\ldots,w_p)$, where
	$$w_{t}:=\begin{cases} -u_t & \mbox{ if $t\in U$}, \\ u_t & \mbox{ if $t\notin U$}. \end{cases}$$
Let $\Delta_{p,F}$ denote the matrix with Gram representation $(w_1,w_2,\ldots,w_p)$.  
Since $F^c=\delta(U)$ is a cutset, negating all the vectors $u_t$ with $t\in U$ results in $(\Delta_{p,F})_{i,i+1}=-1$ for every $i\in F^c$, and all other entries of $\Delta_{p,F}$ remain the same as those in $\Delta_{p,[p]}$.  
Moreover, $\rank(\Delta_{p,F})=\rank(\Delta_{p,[p]})$ and  $\rank(U_{\nonedge})={p-1 \choose 2}-1$.  
Thus, $\Delta_{p,F}$ is extremal in $\symG$ with rank $p-2$.  

Now suppose that $p$ is even.
Fix $i\in[p]$ and consider the $(p-2)$-dimensional Gram representation $(u_1,u_2,\ldots,u_p)$ defined above for the extremal matrix $\Delta_{p,[p]\backslash\{i\}}$.  
Partition the collection of odd subsets of $[p]$ into two blocks, $A$ and $B$, where $A$ consists of all odd subsets of $[p]$ containing $i$.  
Let $F\subset[p]$ be of odd cardinality, and suppose first that $F\in A$.  
Consider the even cardinality subset $M=F^c\cup\{i\}$.  
Thinking of each $i$ in $[p]$ as corresponding to the edge $\{i,i+1\}\in E$, it follows that $M=\delta(U)$ for some cut $(U,U^c)$ of $G$.
Once more, thinking of $U\subset[p]=V(G)$, set
	$$w_{t}:=\begin{cases} -u_t & \mbox{ if $t\in U$}, \\ u_t & \mbox{ if $t\notin U$}, \end{cases}$$
and let $\Delta_{p,F}$ denote the matrix with Gram representation $(w_1,w_2,\ldots,w_p)$.  
Since $M$ is a cutset, it follows that $(\Delta_{p,F})_{s,s+1}=-(\Delta_{p,[p]\backslash\{i\}})_{s,s+1}$.  
In particular, $(\Delta_{p,F})_{i,i+1}=1$.

Finally, suppose $F\in B$, and consider the even cardinality subset $M=F^c\backslash\{i\}$.  
Proceeding as in the previous case produces the desired matrix $\Delta_{p,F}$.  
Just as in the odd case, the matrices $\Delta_{p,F}$ for $p$ even are extremal of rank $p-2$.  
\end{proof}

\begin{example}
\label{lower cardinality T example}
We illustrate the construction in the proof of Lemma \ref{lower cardinality T case} by considering the case $p=4$ and $i=1$.  
The corresponding maximum cardinality subset is $\{2,3,4\}$.  
The $(p-2)$-dimensional Gram representation for this maximum cardinality odd subset is $(u_1,u_2,u_3,u_4)$,
where
$$
u_1:=\begin{bmatrix}
			0 \\
			1\\
			\end{bmatrix}  
\qquad
u_{2}:=\begin{bmatrix}
			1 \\
			-1\\
			\end{bmatrix} 
 \qquad
u_{3}:=\begin{bmatrix}
			1 \\
			0\\
			\end{bmatrix} 
 \qquad
u_{4}:=\begin{bmatrix}
			1 \\
			1\\
			\end{bmatrix}.
$$
The resulting extremal matrix in $\symG$ is 
$$
\Delta_{4,\{2,3,4\}}=\begin{bmatrix}
1	&	-1	&	0	&	1	\\
-1	&	2	&	1	&	0	\\
0	&	1	&	1	&	1	\\
1	&	0	&	1	&	2	\\
\end{bmatrix}.
$$

Now consider the odd cardinality subset $F:=\{2\}\subset[4]$.  
Then $M = F^c\backslash\{1\}=\{3,4\}\simeq\{\{3,4\},\{4,1\}\}\subset E(C_4)$.  
Thus, $M=\delta(U)$ where $U=\{4\}$.  
The Gram representation identified in the proof of Lemma~\ref{lower cardinality T case} is $(w_1,w_2,w_3,w_4):=(u_1,u_2,u_3,-u_4)$.  
Both of these Gram representations are depicted in Figure~\ref{fig: lower cardinality example}.
	\begin{figure}
	\centering
	\includegraphics[width=0.9\textwidth]{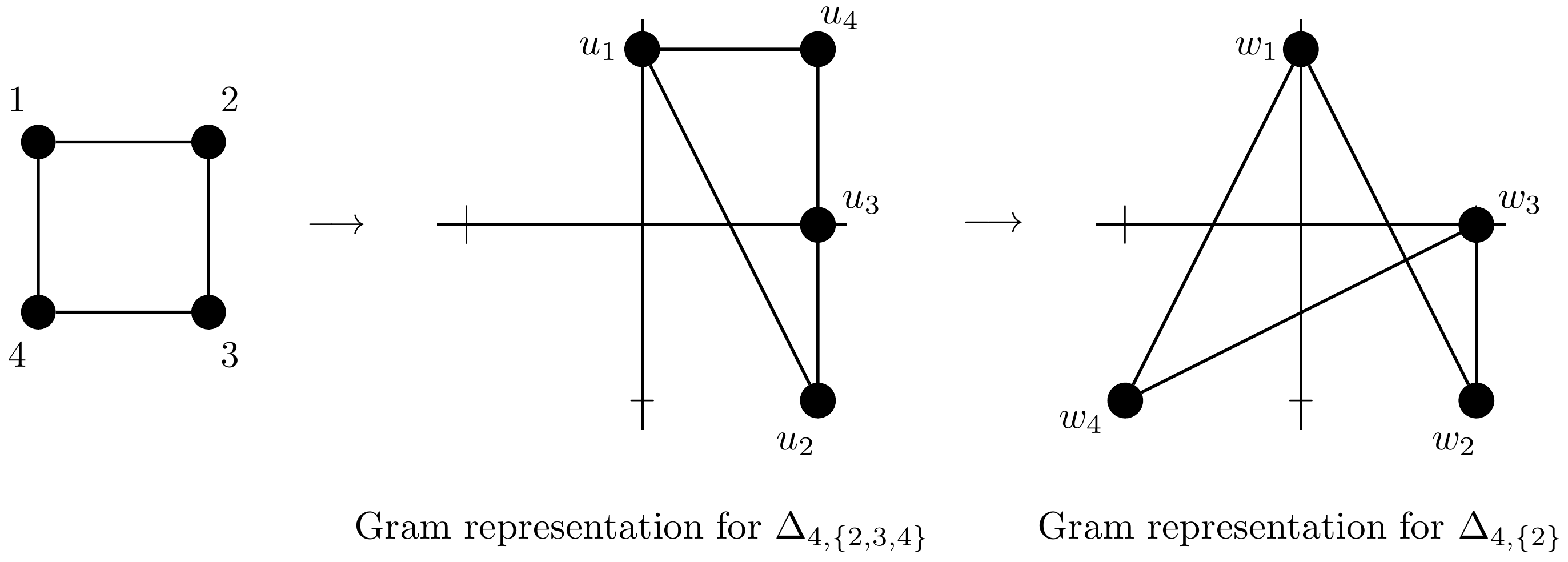}
	\caption{$C_4$ included into $[-1,1]^2$ via its Gram representation, and the Gram representations for Example \ref{lower cardinality T example}.}
	\label{fig: lower cardinality example}
	\end{figure}
The resulting extremal matrix associated to $F$ is
$$
\Delta_{4,\{2\}}=\begin{bmatrix}
1	&	-1	&	0	&	-1	\\
-1	&	2	&	1	&	0	\\
0	&	1	&	1	&	-1	\\
-1	&	0	&	-1	&	2	\\
\end{bmatrix}.
$$
Notice that the off-diagonal entries corresponding to the edges of $C_4$ are given by 
$$
-v^F=(-1,1,-1,-1)\in\R^{E(C_4)},
$$
the normal vector to the facet-supporting hyperplane $\langle v^F,x\rangle = 2$ of $\cutpm{C_4}$. \hfill$\square$
\end{example}

Lemmas \ref{maximal cardinality T case} and \ref{lower cardinality T case} combined provide a proof of Theorem \ref{simplicial facets correspondence thm}:


\subsubsection*{Proof of Theorem \ref{simplicial facets correspondence thm}}
Recall that we identify $\R^E\simeq\R^p$ by identifying the coordinate $e=\{i,i+1\}$ in $\R^E$ with the coordinate $i$ in $\R^p$.  
Consider the projection map $\pi_{E}:\s^p\longrightarrow\R^E\simeq\R^p$ that projects a matrix onto its coordinates corresponding to the edges of $G$.  
For an odd cardinality subset $F$ of $[p]$, the matrix $\Delta_{p,F}$ satisfies $\pi_E(\Delta_{p,F})=-v^F$.  
This completes the proof of Theorem \ref{simplicial facets correspondence thm}. \hfill$\square$

\bigskip

Proposition \ref{halfcubical matrices} and Theorem \ref{simplicial facets correspondence thm} combine to prove that the $p$-cycle has the facet-ray identification property.
We now use these results to provide a proof of Theorem~\ref{main theorem graphs with no K_5 minors}.  


\subsubsection*{Proof of Theorem \ref{main theorem graphs with no K_5 minors}}
Let $G$ be a graph without $K_5$ minors.  
To show that $G$ has the facet-ray identification property, we must produce for every facet $F$ of $\cutpm{G}$ an extremal matrix $M\in \symG$ whose off-diagonal entries are given by the normal vector to $F$.
Recall from Section~\ref{preliminaries} that the supporting hyperplanes of $\cutpm{G}$ are 
\begin{enumerate}[(1)]
	\item $\langle \pm v^e, x \rangle=1$ for all $e\in E$, and
	\item $\langle v^F, x\rangle = m-2$ for all odd cardinality subsets $F\subset E(C_m)$ for all chordless cycles $C_m$ in $G$.  
\end{enumerate} 
In the case of the cycle $C_m$ we have constructed the desired extremal matrices $X_e$, $X_e^-$, and $\Delta_{m,F}$ for each such hyperplane, and each such matrix possesses an underlying Gram representation $(u_1,\ldots,u_m)$.  
Thus, we define the $(m-2)$-dimensional Gram representation $(w_1,w_2,\ldots,w_{|E|})$ where
$$
w_{t}:=
\begin{cases}
 u_t & \mbox{ if $t\in[m]\simeq V(C_m)\subset [p]=V(G)$}, \\
 0 & \mbox{ otherwise}. 
\end{cases}
$$
Let $\widetilde X_e$, $\widetilde X_e^-$, and $\widetilde \Delta_{m,F}$ denote the resulting matrices in $\symG$ with Gram representation $(w_1,\ldots,w_{|E|})$.  
It follows from Proposition~\ref{halfcubical matrices}, Lemma~\ref{maximal cardinality T case} and Lemma~\ref{lower cardinality T case} that these matrices are extremal in $\symG$ of rank $1$, $1$, and $m-2$, respectively.  
This completes the proof of Theorem~\ref{main theorem graphs with no K_5 minors}.\hfill$\square$


\subsection{The Geometry of Facet-ray Identification Revisited}
\label{the geometric correspondence of facet-ray identification}

In Section~\ref{the dual of an elliptope} we saw that the vertices of the polar $\E_G^\circ$ lie on the extremal rays of the cone $\symG$.  
In the polar, this means that an extremal ray of $\symG$ corresponds to a hyperplane supporting $\E_G$.  
Since the extremal rays of $\symG$ are the dimension $1$ faces of the cone, we say that the \emph{rank} of an extremal ray $r$ of $\symG$ is the rank of any nonzero matrix lying on $r$.   
Thus, the rank of an extremal ray of $\symG$ is given by the rank of the corresponding vertex of $\E_G^\circ$. 
In the polar, the \emph{rank} of a supporting hyperplane of $\E_G$ is the rank of the corresponding extremal ray in $\symG$.

Let $G$ be a graph without $K_5$ minors. 
For each facet of $\cutpm{G}$ we have identified an extremal matrix $X_e$, $X_e^-$, or $\Delta_{m,F}$, and each such matrix generates an extremal ray of $\symG$:
$$
r_e:=\spn_{\geq0}(X_e), 
\qquad 
r_e^-:=\spn_{\geq0}(X_e^-),
\qquad 
\mbox{and} 
\qquad
r_{m,F} :=\spn_{\geq0}(\Delta_{m,F}),
$$
respectively.  
Recall from Theorem~\ref{dual body theorem} that $\E_G^\circ$ is a projection of  the trace two affine section of the cone $\symG$.  
Since $\tr(X_e)=\tr(X_e^-)=2$ these matrices correspond to vertices of $\E_G^\circ$, which dually correspond to the facet-supporting hyperplanes $\langle \pm v^e,x\rangle =1$ of the elliptope $\E_G$.  
On the other hand, $\tr(\Delta_{m,F})=\sum_{t=1}^mu_t^Tu_t=3m-6.$
Thus, the matrix $Y_{m,F}:=\frac{2}{3m-6}\Delta_{m,F}$ corresponds to the regular extreme point $\pi_E(Y_{m,F})=-\frac{2}{3m-6}v^F$ of $\E_G^\circ$.
Hence, the corresponding hyperplane in $\E_G$ is 
$$
\langle v^F,x \rangle=\frac{6-3m}{2}.
$$
So the supporting hyperplane $\langle v^F,x\rangle = m-2$ of $\cutpm{G}$ is a translation by $\frac{5m-10}{2}$ of this rank $m-2$ hyperplane.  
This illustrates the geometry described in Section~\ref{the geometric correspondence}.  

\begin{rmk}
\label{holds for all FRIP graphs}
Note that the geometric correspondence between facets and extremal rays discussed in this section holds for any graph with the facet-ray identification property.
Thus, while our proof of this property is combinatorial, the property itself is inherently geometric.
\end{rmk}


\section{Characterizing Extremal Ranks}
\label{characterization of all extremal ranks for series-parallel graphs}
In this section, we discuss when facet-ray identification characterizes all extreme ranks of $\symG$.  

\subsection{Series-parallel graphs.}
\label{sec:series-parallel}
Let $G$ be a series-parallel graph.  
We show that the extremal ranks identified by the facets of $\cutpm{G}$ are \emph{all} the possible extremal ranks of $\symG$, thereby completing the proof of Theorem~\ref{main theorem characterization of all extremal ranks}.  
To do so, we consider the dual cone of $\symG$, namely the cone of all PSD-completable matrices, which we denote by $\CC_G$.  
Recall that a \emph{(real) $p\times p$ partial matrix} $A=[a_{ij}]$ is a matrix in which some entries are specified real numbers and the remainder are unspecified.  
It is called \emph{symmetric} if all the specified entries satisfy $a_{ij}=a_{ji}$, and it is called \emph{PSD-completable} if there exists a specification of the unknown entries of $A$ that produces a matrix $\widetilde A\in\symmetricmatrices$ that is positive semidefinite.  
It is well-known that the dual cone to $\symG$ is the cone $\CC_G$ of all PSD-completable matrices.   
Let $H$ be an induced subgraph of $G$, and let $A[H]$ denote the submatrix of $A\in\s^m$ whose rows and columns are indexed by the vertices of $H$.  
A symmetric partial matrix $A$ is called \emph{(weakly) cycle-completable} if the submatrix $A[C_m]\in\s^m$ is PSD-completable for every chordless cycle $C_m$ in $G$.

\subsubsection*{Proof of Theorem~\ref{main theorem characterization of all extremal ranks}}
By Theorem~\ref{main theorem graphs with no K_5 minors}, $G$ has the facet-ray identification property, and the extreme matrices in $\symG$ identified by the facets of $\cutpm{G}$ are of rank $1$ and $m-2$, where $m$ varies over the length of all chordless cycles in $G$.  
So it only remains to show that these are all the extremal ranks of $\symG$.  
To do so, we consider the dual cone to $\symG$.

In \cite{BJL96} it is shown that a symmetric partial matrix $A$ is in the cone $\CC_G$ if and only if $A$ is cycle completable.  
Since $\CC_G$ is the dual cone to $\symG$ it follows that $A\in\CC_G$ if and only $\langle A,X\rangle\geq0$ for all $X\in \symG$.  
Applying this duality, we see that the matrix $A$ satisfies
$
\langle A,X\rangle\geq 0
$
 for all $X\in\symG$
if and only if
$
\langle A[C_m],X\rangle\geq 0 
$
 for all extremal matrices $X\in\s^m_{\succeq0}(C_m)$ for all chordless cycles $C_m$ in $G$.  
Here, we think of the matrices $A[C_m]$ and $X\in\s^m_{\succeq0}(C_m)$ as living in $\s^p$ by extending the matrices $A[C_m]$ and $X$ in $\s^{V(C_m)}$ by placing zeros in the entries corresponding to edges not in the chordless cycle $C_m$.  
It follows from this that the cone $\CC_G$ is dual to $\symG$ and the cone whose extremal rays are given by the chordless cycles in $G$.  
Thus, these two cones must be the same, and we conclude that the only possible ranks of the extremal rays of $\symG$ are those given by the ranks of $\s^m_{\succeq0}(C_m)$ as $C_m$ varies over all chordless cycles in $G$.  
This completes the proof of Theorem~\ref{main theorem characterization of all extremal ranks}.  \hfill$\square$

\subsection{Some further examples.}
Theorem~\ref{main theorem characterization of all extremal ranks} provides a subcollection of the graphs with no $K_5$ minors for which the facets of $\cutpm{G}$ characterize all extremal ranks of $\symG$, namely those which also have no $K_4$ minors.  
It is then natural to ask whether or not the extreme ranks of the graphs with $K_4$ minors but no $K_5$ minors are characterized by the facets as well.  
The following two examples address this issue.  
Example~\ref{k-3-3 example} is an example of a graph $G$ with a $K_4$ minor but no $K_5$ minor for which the facets do not characterize all extremal ranks of $\symG$, and Example~\ref{k-4 minor no k-3-3 minor example} is an example of a graph $G$ with a $K_4$ minor but no $K_5$ minor for which the extremal ranks of $\symG$ are characterized by the facets of $\cutpm{G}$.  

\begin{example}
\label{k-3-3 example}
Consider the complete bipartite graph $G:= K_{3,3}$. 
In \cite{grone} the extremal rays of $\symG$ are characterized, and it is shown that $G$ has extremal rays of ranks $1,2,$ and $3$.
However, with the help of {\tt Polymake} \cite{polymake} we see that the facet-supporting hyperplanes of $\cutpm{G}$ are $x_e=\pm1$ for each edge $e\in E(G)$ together with $\langle v^F,x\rangle = m-2$ as $C_m$ varies over the nine (chordless) $4$-cycles within $G$.  
Thus, the constant terms of the facet-supporting hyperplanes only capture extreme ranks $1$ and $2$, but not $3$.  
\end{example}

	\begin{figure}
	\centering
	\begin{tabular}{c c}
	\includegraphics[width=0.2\textwidth]{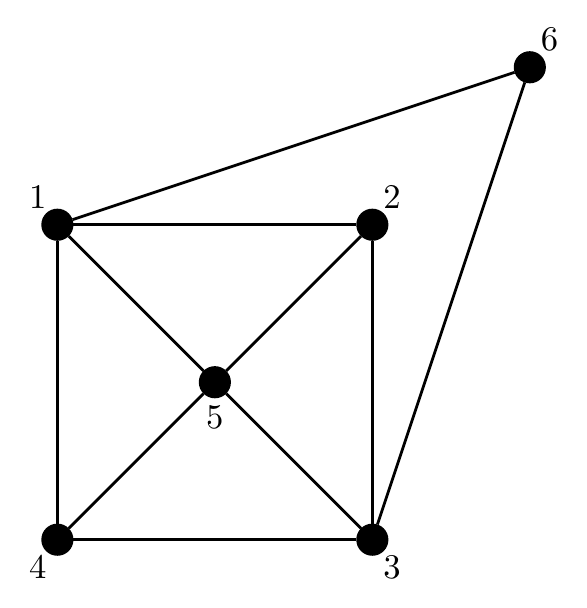}	\qquad&\qquad	\includegraphics[width=0.35\textwidth]{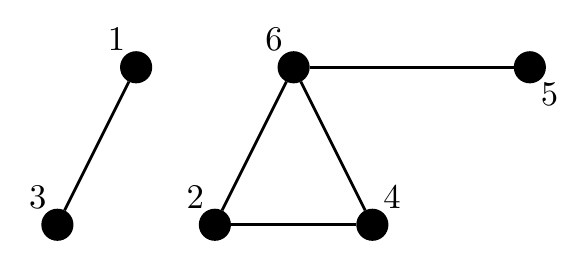}	\\
	$G$\qquad		&\qquad	$G^c$	\\
	\end{tabular}
	\caption{The graph $G$ from Example~\ref{k-4 minor no k-3-3 minor example} and its complement $G^c$.}
	\label{fig: k-4 minor no-k-3-3 minor}
	\end{figure}

\begin{example}
\label{k-4 minor no k-3-3 minor example}
Consider the graph $G$ depicted in Figure~\ref{fig: k-4 minor no-k-3-3 minor}. Recall that a $k$-block is a graph $P$ of order $k$ that has no proper induced subgraph of order $k$.  
Agler et al.~characterized all $3$-blocks in \cite[Theorem 1.5]{agler} in terms of their complements.  
It follows immediately from this theorem that $G$ contains no induced $3$-block.  
Thus, $\ord(G)\leq 2$, and since $G$ is not a chordal graph we see that $\ord(G) = 2$.  
By Theorem~\ref{main theorem characterization of all extremal ranks} the facets of $\cutpm{G}$ identify extremal rays of rank $1$ and $2$.  
Thus, all possible extremal ranks of $G$ are characterized by the facets of $\cutpm{G}$.  
\end{example}

The reader may also notice that the graph $G$ from Example~\ref{k-4 minor no k-3-3 minor example} also no $K_{3,3}$ minor, while the graph from Example~\ref{k-3-3 example} is $K_{3,3}$.  
Thus, it is natural to ask if the collection of graphs for which the facets characterize the extremal ranks of $\symG$ are those with no $K_{3,3}$ minor.  
The following example shows that this is not the case. 

\begin{example}
\label{k-4 minor no k-3-3 minor not characterized example}
Consider the following graph $G$ and its complement $G^c$: 
	\begin{center}
	\centering
	\begin{tabular}{c c}
	\includegraphics[width=0.2\textwidth]{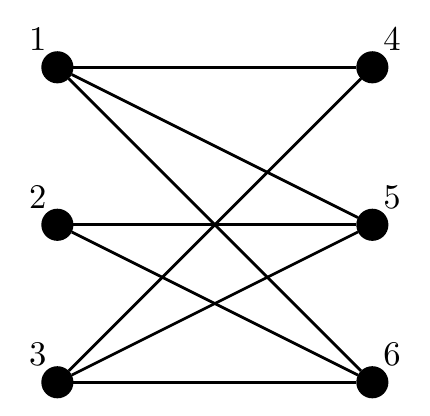}	\quad\qquad&\quad\qquad	\includegraphics[width=0.35\textwidth]{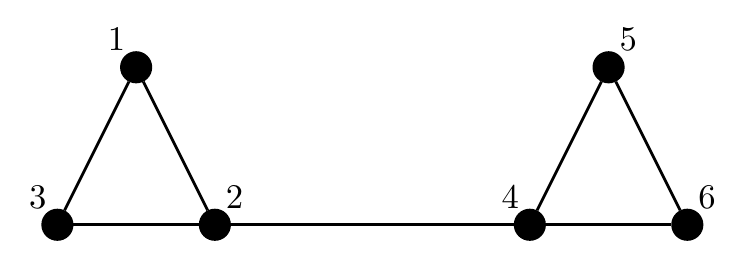}	\\
	$G$		\quad\qquad&\qquad\quad	$G^c$	\\
	\end{tabular}
	\label{fig: k-4 minor no k-3-3 minor}
	\end{center}
Notice that $G$ contains no $K_{3,3}$ minor, but it does contain a $K_4$ minor.  
By \cite[Theorem 1.5]{agler} $G$ is a $3$-block since its complement graph is two triangles connected by an edge.  
Thus, $G$ has an extremal ray of rank $3$, but by Theorem~\ref{main theorem graphs with no K_5 minors} the facets of $\cutpm{G}$ only detect extremal rays of ranks $1$ and $2$.  
\end{example}

Examples~\ref{k-3-3 example}, \ref{k-4 minor no k-3-3 minor example}, and \ref{k-4 minor no k-3-3 minor not characterized example} together show that describing the collection of graphs for which the facets of $\cutpm{G}$ characterize the extremal ranks of $\symG$ is more complicated that forbidding a particular minor.
Indeed, the collection of graphs with this property is not even limited to the graphs with no $K_5$ minors, as demonstrated by Example~\ref{graph with k-5 minor example}.

\begin{example}
\label{graph with k-5 minor example}
Consider the graph $G$ depicted in Figure~\ref{fig: graph with k-5 minor}.
	\begin{figure}
	\centering
	\begin{tabular}{c c}
	\includegraphics[width=0.3\textwidth]{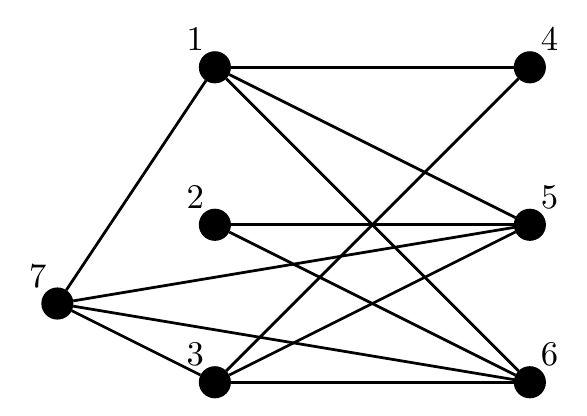}	\quad\qquad&\qquad\quad	\includegraphics[width=0.35\textwidth]{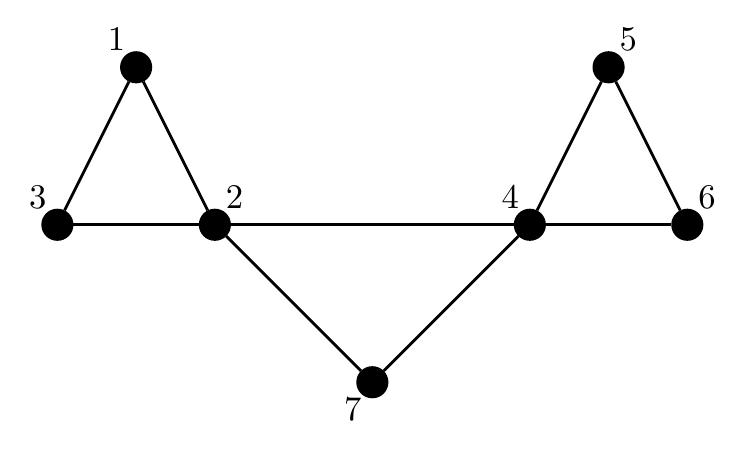}	\\
	$G$		\quad\qquad&\qquad\quad	$G^c$	\\
	\end{tabular}
	\caption{A graph with a $K_5$ minor whose facets characterize all extremal rays.}
	\label{fig: graph with k-5 minor}
	\end{figure}
To see that this graph has the facet-ray identification property we first compute the $114$ facets of $\cutpm{G}$ using {\tt Polymake} \cite{polymake}.  
The resulting computation yields $72$ cycle inequalities, 16 for the four $3$-cycles, and $56$ for the seven chordless $4$-cycles, as well as eight inequalities for the four edges not in a $3$-cycle.  
These $80$ facets identify extremal rays of rank $1$ and $2$ just as in the case of the graphs with no $K_5$ minors.  
The remaining $64$ facet-supporting inequalities of $\cutpm{G}$ are given by applying the \emph{switching operation} defined in \cite[Chapter 27]{deza} to the inequality
$$
x_{14}-x_{15}-x_{34}-x_{36}-x_{37}-x_{67}+x_{16}+x_{17}+x_{25}+x_{26}+x_{35}+x_{57}\leq 4.
$$
This new collection of facets identifies extremal rays of rank $3$.  
For example, the presented inequality specifies the off-diagonal entries of the following rank $3$ matrix:
$$
\begin{pmatrix}
2	&	0	&	0	&	1	&	1	&	-1	&	-1	\\

0	&	1	&	0	&	0	&	-1	&	-1	&	0	\\

0	&	0	&	2	&	1	&	-1	&	1	&	1	\\

1	&	0	&	1	&	1	&	0	&	0	&	0	\\

1	&	-1	&	-1	&	0	&	2	&	0	&	-1	\\

-1	&	-1	&	1	&	0	&	0	&	2	&	1	\\

-1	&	0	&	1	&	0	&	-1	&	1	&	1	\\
\end{pmatrix}.
$$
This matrix has the $3$-dimensional Gram representation
$$
u_1 = 
\begin{pmatrix}
1	\\
1	\\
0	\\
\end{pmatrix},
u_2 = 
\begin{pmatrix}
0	\\
0	\\
-1	\\
\end{pmatrix},
u_3 = 
\begin{pmatrix}
1	\\
-1	\\
0	\\
\end{pmatrix},
u_4 = 
\begin{pmatrix}
1	\\
0	\\
0	\\
\end{pmatrix},
u_5 = 
\begin{pmatrix}
0	\\
1	\\
1	\\
\end{pmatrix},
u_6 = 
\begin{pmatrix}
0	\\
-1	\\
1	\\
\end{pmatrix},
u_7 = 
\begin{pmatrix}
0	\\
-1	\\
0	\\
\end{pmatrix}.
$$
It follows via an application of Theorem~\ref{agler thm} that this matrix is extremal in $\symG$.  
Similar matrices can be constructed for each of the $64$ facets of this type.  
Thus, $G$ has the facet-ray identification property, and the facets identify extreme rays of rank $1,2,$ and $3$.  

To see that these are all of the extremal ranks of $\symG$ recall from Section~\ref{preliminaries} that since $G$ has $7$ vertices then $\ord(G)\leq 5$ with equality if and only if $G$ is the cycle on $7$ vertices.  
Thus, it only remains to show that $\ord(G)\neq 4$.  
 To see this, we examine the complement of $G$ depicted in Figure~\ref{fig: graph with k-5 minor}.  
 By \cite[Theorem 0.2]{HLW} $G$ is not a $4$-superblock since the complement of $G$ can be obtained by identifying the vertices of the graphs
 \begin{center}
\includegraphics[width=0.3\textwidth]{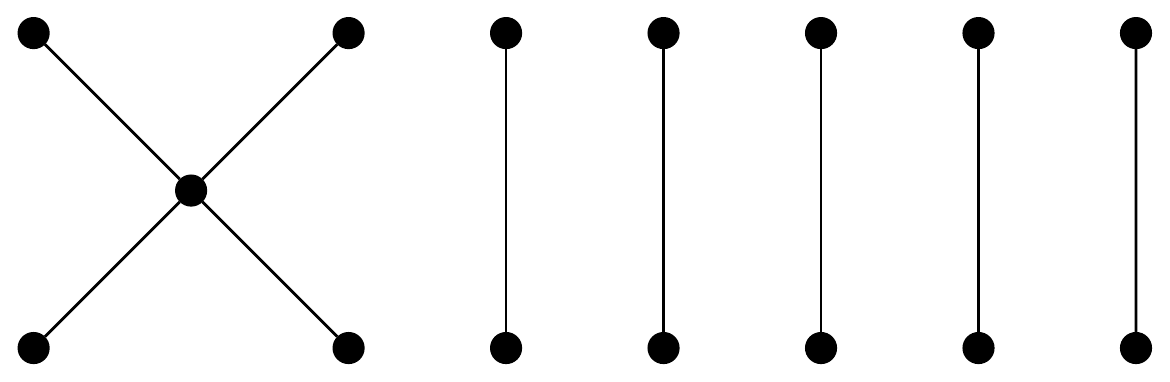}
 \end{center}
 Thus, if $G$ has rank $4$ extremal rays then it must contain an induced $4$-block.  
 However, one can check that all induced subgraphs either have order $1$, $2$, or $3$.  
 Therefore, $G$ is a graph with a $K_5$ minor that has the facet-ray identification property and for which the extremal ranks of $\symG$ are characterized by the facets of $\cutpm{G}$.  
 Moreover, this example shows that the types of facets which identify extreme rays of $\symG$ are not limited to those arising from edges and chordless cycles.  
\end{example}

We end this section with a problem presented by these various examples.

\begin{prob}
\label{characterization problem}
Determine all graphs $G$ with the facet-ray identification property for which the facets of $\cutpm{G}$ characterize all extremal ranks of $\symG$.  
\end{prob}


\section{Graphs Without the Facet-Ray Identification Property}
\label{graphs without the facet-ray identification property}

In the previous sections we discussed various graphs $G$ which have the facet-ray identification property.  
Here, we provide an explicit example showing that not all graphs admit the facet-ray identification property.  

\begin{example}
\label{parachute example}
Consider the \emph{parachute graph} on $7$ vertices depicted in Figure~\ref{fig: parachute graph}.  
The parachute graphs on $2k+1$ vertices for $k\geq1$ are defined in \cite{deza}.
	\begin{figure}
	\centering
	\includegraphics[width=0.3\textwidth]{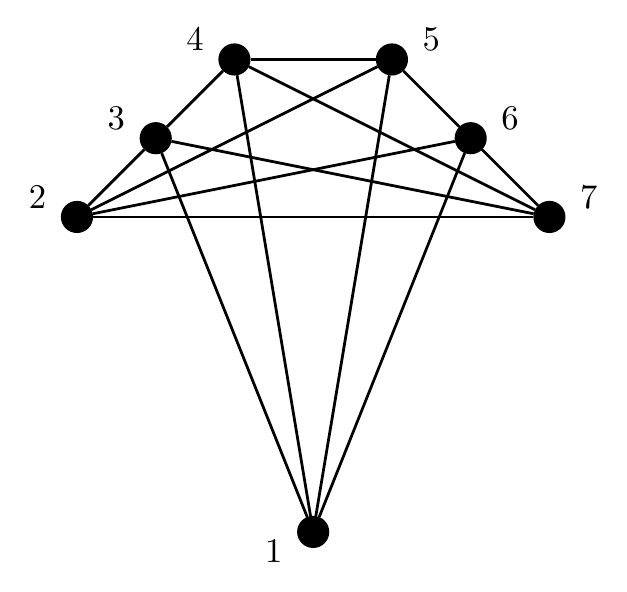}
	\caption{The parachute graph on $7$ vertices.}
	\label{fig: parachute graph}
	\end{figure}
Using {\tt Polymake}~\cite{polymake} we compute the facets of $\cutpm{G}$ and find that 
$$
x_{13}+x_{14}+x_{15}+x_{16}+x_{25}+x_{26}+x_{27}+x_{37}+x_{47}-x_{23}-x_{34}-x_{45}-x_{56}-x_{67}\leq 4
$$ 
is a facet-defining inequality.  
Thus, if $G$ has the facet ray identification property there exists a filling of the partial matrix 
$$
M := 
\begin{pmatrix}
x_1	&	0	&	1	&	1	&	1	&	1	&	0	\\

0	&	x_2	&	-1	&	0	&	1	&	1	&	1	\\

1	&	-1	&	x_3	&	-1	&	0	&	0	&	1	\\

1	&	0	&	-1	&	x_4	&	-1	&	0	&	1	\\

1	&	1	&	0	&	-1	&	x_5	&	-1	&	0	\\

1	&	1	&	0	&	0	&	-1	&	x_6	&	-1	\\

0	&	1	&	1	&	1	&	0	&	-1	&	x_7	\\
\end{pmatrix}.
$$
that results in a positive semidefinite matrix which is extremal in $\symG$.  
Notice that the minimum rank of a positive semidefinite completion of $M$ is $5$.  
To see this, recall that if the $\rank(M)<5$ then the point $(x_1,x_2,\ldots,x_7)$ must lie on the variety of the ideal $I$ generated by the $5\times 5$ minors of $M$.  
Using Macaulay2, we see that the minimal generating set for the ideal $I$ includes the generator $x_1+x_2+\ldots+x_7+10$.  
If $M$ is positive semidefinite then $x_i\geq0$ for all $1\leq i\leq 7$, and so $(x_1,\ldots,x_7)$ cannot be a point in the variety of the ideal $I$.  

On the other hand, the maximum dimension of the frame space 
$$
\spn_\R\left(U_{\nonedge}\right) = \spn_\R(u_iu_j^T+u_ju_i^T : ij\in\nonedge)
$$
for any $k$-dimensional Gram representation of $G$ is at most the number of nonedges of $G$, which is seven.  
By Theorem~\ref{agler thm}, since $7<{5+1\choose 2}-1$ no positive semidefinite completion of $M$ can be extremal in $\symG$. 
Thus, $G$ does not have the facet-ray identification property.  
\end{example}

The facet-defining inequality considered in Example~\ref{parachute example} has been studied before as a facet-defining inequality of the cut polytope of the complete graph $K_7$ by Deza and Laurent \cite{deza}, and is referred to as a \emph{parachute inequality}.  
Thus, one consequence of the above example is that $K_7$ also does not have the facet-ray identification property, nor does any $G$ for which the above inequality is facet-defining.  
This suggests that one way to determine the collection of graphs which have the facet-ray identification property is to study those facets which can never identify an extremal matrix in $\symG$.

\begin{prob}
\label{forbidden facet problem}
Determine facet-defining inequalities of $\cutpm{G}$ that can never identify extremal matrices in $\symG$. 
\end{prob}


\section*{Acknowledgements}
We wish to thank Alexander Engstr\"om and Bernd Sturmfels for various valuable discussions and insights.  CU was partially supported by
the Austrian Science Fund (FWF) Y 903-N35.


\end{document}